\documentclass{amsart}
\usepackage[foot]{amsaddr}
\newtheorem{theorem}{Theorem}
\newtheorem{corollary}{Corollary}[theorem]
\newtheorem{lemma}{Lemma}
\newtheorem{proposition}{Proposition}

\theoremstyle{definition}
\newtheorem{definition}{Definition}
\theoremstyle{remark}
\newtheorem{remark}{Remark}

\usepackage[utf8]{inputenc}
\usepackage{enumerate}
\usepackage{mathrsfs}
\usepackage{amssymb}
\usepackage{mathtools}
\usepackage{graphicx}
\usepackage{hyperref}
\usepackage{soul}
\usepackage{lipsum}

\graphicspath{ {./images/} }

\newcommand{\norm}[1]{\left\lVert #1 \right\rVert}

\newcommand{\innerprod}[2]{\langle #1, #2 \rangle}

\usepackage{stackrel}
\usepackage{color}
\usepackage{url}

\newcommand{\bx}{{ \mathbf{x} }}
\newcommand{\by}{{\mathbf{y}}}

\newcommand{\br}{{\mathbf{r}}}
\newcommand{\bn}{{\mathbf{n}}}

\newcommand{\bu}{{ \mathbf{u}}}
\newcommand{\bU}{{ \mathbf{U}}}
\newcommand{\bv}{{ \mathbf{v}}}
\newcommand{\bw}{{ \mathbf{w}}}
\newcommand{\bV}{{ \mathbf{V}}}

\newcommand{\bS}{{\mathbf{S}}}

\newcommand{\bI}{{\mathbf{I}}}
\newcommand{\bF}{{\bf F}}

\newcommand{\lb}{\label}
\newcommand{\be}{\begin{equation}}
\newcommand{\ee}{\end{equation}}
\newcommand{\bea}{\begin{eqnarray}}
\newcommand{\eea}{\end{eqnarray}}
\newcommand{\beal}{\begin{equation}\begin{aligned}}
\newcommand{\eeal}{\end{aligned}\end{equation}}
\newcommand{\bzed}{{\bf 0}}

\newcommand{\btau}{{\mbox{\boldmath $\tau$}}}

\newcommand{\bomega}{{\mbox{\boldmath $\omega$}}}
\newcommand{\barphi}{{\mbox{\boldmath $\varphi$}}}

\newcommand{\bpsi}{{\mbox{\boldmath $\psi$}}}

\newcommand{\bPsi}{{\mbox{\boldmath $\Psi$}}}

\newcommand{\btimes}{{\mbox{\boldmath $\,\times\,$}}}
\newcommand{\bdot}{{\mbox{\boldmath $\,\cdot\,$}}}
\newcommand{\bdots}{{\mbox{\boldmath $\,:\,$}}}
\newcommand{\grad}{{\mbox{\boldmath $\nabla$}}}
\newcommand{\ext}{\textbf{Ext}}

\newcommand{\red}[1]{\textcolor{black}{#1}}

\begin{document}

\title[Weak-Strong Uniqueness]{Weak-Strong Uniqueness \\ and the d'Alembert Paradox}

\author[H. Quan]{Hao Quan$^1$}

\email{$^1$haoquan@jhu.edu}

\author[G. Eyink]{Gregory L. Eyink$^2$}
\email{$^2$eyink@jhu.edu}

\address{$^1, ^2$Department of Applied Mathematics \& Statistics The Johns Hopkins University, Baltimore, MD 21218, USA}
\address{$^2$Department of Physics and Astronomy The Johns Hopkins University, Baltimore, MD 21218, USA}


\keywords{}

\date{\today}

\dedicatory{}

\begin{abstract}
We prove conditional weak-strong uniqueness of the potential Euler solution for external flow around a smooth body in three space dimensions, within the class of viscosity weak solutions with the same initial data. Our sufficient condition is the vanishing of the 
streamwise component of the skin friction \red{integrated over the surface}
in the inviscid limit, \red{slightly stronger than the condition of Kelliher} 
\red{and weakening that of Bardos-Titi}, both for bounded domains. Because global-in-time existence of the smooth potential solution leads back to the d'Alembert paradox, we argue that weak-strong uniqueness is not a valid criterion for ``relevant'' notions of generalized Euler solution and that our condition is likely to be violated in the inviscid limit. 
We prove also that the Drivas-Nguyen condition on uniform continuity at the wall of the normal velocity component implies weak-strong uniqueness within the general class of admissible weak Euler solutions in bounded domains.  

\smallskip
\noindent \textbf{Keywords}: Weak-strong uniqueness, D'Alembert paradox, inviscid limit, dissipative weak Euler solution 
\end{abstract}
\maketitle

\section{Introduction}

The concept of weak-strong uniqueness in the theory of partial differential equations (PDE's) arose 
in the work of Leray \cite{leray1934mouvement}, Prodi \cite{prodi1959teorema}, and Serrin \cite{Serrin1963} 
for the incompressible Navier-Stokes equations. Weak-strong uniqueness for a PDE can be expressed 
by the statement that ``If there exists a strong solution, then any weak solution with 
the same initial data coincides with it'',  as succinctly summarized in the recent review \cite{wiedemann2018weak}. 
This same review also emphasized the important role that weak-strong uniqueness has come to play in 
the theory of incompressible Euler equations, especially in the formulation of ``relevant'' notions 
of generalized solutions. Indeed, standard weak or distributional solutions of Euler equations need  
not arise as inviscid limits of Navier-Stokes solutions, so that more general
notions have been proposed, such as the ``measure-valued Euler solutions'' of DiPerna-Majda\cite{diperna1987oscillations}. 
While these  measure-valued solutions are guaranteed to exist as inviscid limits, Lions \cite{lions1996mathematical} 
in particular was critical of them, arguing that ``the relevance of this notion is not entirely clear since 
it is not known that `solutions' in the sense of \cite{diperna1987oscillations} coincide with smooth solutions 
as long as the latter do exist.'' Lions \cite{lions1996mathematical} thus proposed his own notion of 
``dissipative Euler solutions'', which are likewise guaranteed to exist as inviscid limits but which were designed
to have in addition the weak-strong uniqueness property. Lions' theory has had important applications 
to turbulence theory, providing a proof that finite-time blow-up of smooth Euler solutions is necessary to 
explain anomalous energy dissipation that might arise from smooth initial data in three-dimensional periodic domains\cite{eyink2008dissipative,bardos2013mathematics}, for example the Taylor-Green vortex initial data 
\cite{fehn2022numerical}. 

Weak-strong uniqueness cannot hold unconditionally, as shown already by the early examples 
of non-unique weak Euler solutions constructed by Scheffer \cite{scheffer1993inviscid} and Shnirelman 
\cite{shnirelman1997nonuniqueness} with compact space-time support. The modest ``admissibility condition''
\be \frac{1}{2}\int_\Omega |\bu(\bx,t)|^2 \,dV\leq \frac{1}{2}\int_\Omega |\bu(\bx,0)|^2 \, dV, 
\quad t\geq 0 \lb{admiss} \ee 
assures that any standard weak Euler solution for $\Omega={\mathbb R}^n$ or ${\mathbb T}^n$ with $n\geq 2$ 
is a dissipative solution in the sense of Lions and thus satisfies weak-strong uniqueness 
\cite{lions1996mathematical,eyink2008dissipative,bardos2013mathematics}. In fact, a generalization of this simple 
admissibility condition has been shown to imply weak-strong uniqueness also for measure-valued 
Euler solutions on space domain $\Omega={\mathbb R}^n$ or ${\mathbb T}^n$ \cite{brenier2011weak}. 
The situation is not as simple for domains with a non-empty boundary, $\partial\Omega\neq\emptyset.$
Using convex integration methods, a piecewise smooth, stationary Euler solution in a 2D annular domain 
was shown to co-exist with infinitely-many admissible weak Euler solutions for the same initial data 
\cite{bardos2014non}. More recently, similar methods were applied to show that the analogous result 
holds for plug flow, with space-time constant streamwise velocity $\bU$ in a 3D plane-parallel 
channel, which coexists with infinitely many admissible weak Euler solutions with the same initial 
condition that exhibit separation at the boundary \cite{vasseur2023boundary}. 
It was proved on the other hand by Bardos \& Titi in \cite{bardos2013mathematics}
that weak-strong uniqueness holds in the class of inviscid limits for such wall-bounded flows 
if some additional conditions are assumed, such as vanishing skin friction or the condition of Kato 
\cite{kato1984remarks} on vanishing dissipation in a shrinking neighborhood of the boundary. 
\red{Kelliher in \cite{kelliher2017observations} has shown for inviscid limits in 2D 
and formally in 3D that vanishing streamwise component of skin friction, integrated over
the surface, implies weak-strong uniqueness in bounded domains, and he further relates 
these conditions to those of  Bardos \& Titi in \cite{bardos2013mathematics}.}
More generally,
it was shown in \cite{bardos2014non} that weak Euler solutions in a bounded domain satisfy the weak-strong
uniqueness property if, in addition to the admissibility condition \eqref{admiss}, they possess
also H\"older regularity of class $C^\alpha$ for some $\alpha>0$ in a neighborhood of the boundary.
Paper \cite{wiedemann2018weak} has further reduced this additional requirement for weak-strong 
uniqueness to just continuity in a neighborhood of the boundary. 

There are several possible views about the physical relevance of such conditional weak-strong uniqueness 
results. One view is that these theorems provide additional conditions for ``physical'' weak
Euler solutions in domains with boundaries. However, in our opinion, such a view unjustifiably assumes 
that Nature will prefer a smooth Euler solution, whenever that exists. A possible counterexample is 
the potential Euler solution for flow around a body, which was shown by d'Alembert 
\cite{dalembert1749theoria,dalembert1768paradoxe} to produce no drag.  Substantial empirical evidence 
exists, on the other hand, that drag around solid bodies does not vanish even in the limit of infinite 
Reynolds number \cite{eyink2024onsager}. In particular, the famous problem of an impulsively accelerated 
disk proposed by Prandtl \cite{prandtl1925magnuseffekt} corresponds to solving the Navier-Stokes equations 
with the potential Euler flow of d'Alembert as initial data, but the latest high-Reynolds number simulations 
of \cite{chatzimanolakis2022vortex} show no obvious tendency for the Navier-Stokes solutions to 
converge to the stationary potential Euler flow. If weak-strong uniqueness were to hold, then there
would be a possible contradiction with theorems that guarantee strong convergence of inviscid limits 
to dissipative weak Euler solutions (e.g. see the review in \cite{drivas2019remarks}). It 
has been argued instead in 
\cite{drivas2018nguyen,quan2025weakb} that the more likely scenario is that the conditions 
for weak-strong uniqueness fail for the inviscid limit in wall-bounded flows and that dissipative weak 
Euler solutions obtained as inviscid limits and the smooth potential Euler solution can thus co-exist,
with the same initial data.
Because the smooth potential solution of d'Alembert exists globally in time, there is no 
possibility to explain the observations by finite-time blow-up of the smooth Euler solution, contrary to  
what has often been suggested for periodic domains (\cite{frisch1995turbulence},\S 7.8). An even clearer numerical 
example of the scenario proposed in \cite{quan2025weakb} is provided by the problem of a vortex 
dipole in 2D impinging on a flat wall \cite{orlandi1990vortex}. 
Although smooth Euler solutions exist globally in time in 2D, 
numerical simulations of \cite{nguyenvanyen2018energy} show no tendency for the high Reynolds 
Navier-Stokes solutions to converge to the smooth Euler solution with the same initial data. 
Furthermore, the numerical evidence of \cite{nguyenvanyen2018energy} is consistent with
non-vanishing skin friction and with anomalous energy dissipation near the wall, so that neither of the 
conditions established by Bardos \& Titi in \cite{bardos2013mathematics} for weak-strong uniqueness of inviscid limits appears 
to be valid for this flow. 

The possible paradox on the inviscid limit for an accelerated body is not yet sharp because, 
to our knowledge, no existing theorem on weak-strong uniqueness applies to the d'Alembert flow.
For example, the proofs of \cite{bardos2013mathematics} carry through for flows in exterior 
domains but they consider inviscid limits of Leray weak Navier-Stokes solutions with finite 
total energy, whereas the solutions involved in the accelerated body have infinite energy 
in the rest frame of the body. Our principal goal in this paper is therefore to prove a 
conditional weak strong-uniqueness result for strong inviscid limits of external flow around a body 
with initial data that converges strongly in $L^2_{loc}$ to the potential Euler solution.
The principal tool that we employ in our proof is the Josephson-Anderson relation recently 
derived in \cite{eyink2021josephson,eyink2021jaerratum} for such flows in the body frame 
and rigorously proved in \cite{quan2024onsager} to remain valid for strong inviscid limits. 
Our proof is a version of a standard relative energy argument \cite{wiedemann2018weak}
and it yields weak-strong uniqueness under a condition \red{of vanishing integrated streamwise skin friction, which is the analogue of the condition established by Kelliher \cite{kelliher2017observations} for bounded domains.} 
It was shown in \cite{quan2024onsager} that 
the skin friction in fact vanishes in the sense of distributions under a condition introduced by 
Drivas \& Nguyen in \cite{drivas2018nguyen} to study anomalous energy dissipation, which involves uniform continuity 
of only the normal component of the velocity and only at the boundary itself. This is 
weaker than the continuity conditions invoked in \cite{bardos2014non} and \cite{wiedemann2018weak}
to prove weak-strong uniqueness for admissible weak Euler solutions in bounded domains. 
For comprehensiveness, we prove also that the less restrictive conditions of \cite{drivas2018nguyen} suffice 
to derive the weak-strong uniqueness results of \cite{bardos2014non} and \cite{wiedemann2018weak}. 

In the remainder of this paper we first give precise statements of the theorems outlined above.
Thereafter we present the proofs. For more complete discussion of 
physical context and implications, we refer the reader to \cite{quan2025weakb}.

\section{Statement of the Main Results}\label{sec:prelim} 
Let $\Omega\subset\mathbb{R}^3$ be a domain with a $C^\infty$ boundary $\partial\Omega$. Recall that the incompressible Euler equations are
\begin{equation}
    \begin{aligned}
        \partial_t \bu + \grad\bdot(\bu\otimes\bu) + \grad p &= 0\quad \text{on } \Omega\times(0,T)\\
        \grad\bdot\bu &= 0\quad \text{on } \Omega\times(0,T)
    \end{aligned}\label{eq:euler}
\end{equation}
with initial data
\begin{equation}
    \bu|_{t=0} = \bu_0 \quad \text{in } \Omega
\end{equation}
and no-flow-through boundary condition
\begin{equation}
    \bu\bdot\bn = 0\quad\text{on }\partial\Omega\times(0,T)\label{eq:no-flow-through-cond}
\end{equation}
Here $T>0$ is a finite time, 
$\bu:\Omega\times[0,T)\to\mathbb{R}^3$ is the velocity, 
$p:\Omega\times[0,T)\to\mathbb{R}$ is the pressure, 
$\bu_0$ is the initial velocity, 
and $\bn$ is the outward-pointing unit normal to the boundary of $\Omega$.
In order to distinguish vector functions from scalar functions and to simplify notations, we use boldface symbols to denote the former and omit codomains in the notations for space of vector functions.

Consider flow past a compact and smooth solid body $B\subset\mathbb{R}^3$ but with a smooth far-field velocity $\bV\in C^{\infty}([0,\infty))$ which may vary over time. In this case, the fluid domain $\Omega = \mathbb{R}^3\setminus B$ is unbounded with a compact boundary $\partial\Omega = \partial B$ (see Figure \ref{fig:external-flow}) and the Euler equations \eqref{eq:euler} are supplemented with a condition on the far field asymptotic velocity:
\begin{equation}
    \bu(\bx, t)\sim \bV(t)\quad \text{as } |\bx|\to\infty\label{eq:far-distance-vel}
\end{equation}
Then the potential flow solution $\bu_{\phi} = \grad\phi$ of the Euler equation \eqref{eq:euler} is given by the velocity potential $\phi$, which is the solution of the Neumann problem of the Laplace equation
\beal
    \Delta\phi = 0 \quad &\text{in }\Omega\\
    \frac{\partial\phi}{\partial n} = 0 \quad &\text{on }\partial B\\
    \phi(\bx,t) \sim \bV(t)\bdot\bx \quad &\text{as } |\bx|\to \infty\label{eq:neumann-problem}
\eeal
for any $t\in[0,\infty)$.
By classical theory of elliptic equations, one has $\phi(t)\in C^{\infty}(\bar\Omega)$ unique up to a spatial constant
(see Section 2 of \cite{quan2024onsager}). Therefore, we deduce that $\bu_\phi\in C^{\infty}(\bar\Omega\times[0,T])$. Furthermore, pressure is given by the unsteady Bernoulli equation:
\begin{equation}
    \partial_t\phi + \frac{1}{2}|\bu_\phi|^2 + p_\phi = C(t)\label{eq:bernoulli}
\end{equation}
for some smooth function $C$ which varies over time so that \eqref{eq:euler} hold for $\bu_\phi$. The total force $\mathbf{F}_\phi$ exerted by the potential flow $\bu_\phi$ on the body $B$ is given instantaneously by the surface integral
\begin{equation}
    \mathbf{F}_\phi(t) \coloneqq -\int_{\partial B}p_\phi(t)\bn\;dS\label{eq:total-force}
\end{equation}
where $\bn$ is the outer normal of the body $B$ pointing into the fluid domain $\Omega$.
In the case of d'Alembert \cite{dalembert1749theoria, dalembert1768paradoxe} with constant far-field velocity $\bV(t) = \bV$ for any $t > 0$ and stationary potential flow $\bu_\phi$, the total force on the body vanishes identically, $\bF_\phi\equiv 0$. This result can be generalized to the unsteady potential flow if averaged over a long enough time. Specifically,
\begin{proposition}
    Consider a solid body $B\subset\mathbb{R}^3,$ represented by a simply connected $C^{\infty}$ manifold with vanishing genus/first Betti number and a compact boundary.
    Let $\bu_\phi$ be the unique potential flow solution of the incompressible Euler equations \eqref{eq:euler} in $\Omega=\mathbb{R}^3\setminus B$ that satisfies no-flow-through condition \eqref{eq:no-flow-through-cond} and has velocity $\bV\in C^{\infty}([0,\infty))$ at infinity. If 
    $\bV$ is globally bounded, then the long-time average of the total force $\mathbf{F}_\phi$ given by \eqref{eq:total-force} must vanish
    \begin{equation}
        \langle\mathbf{F}_\phi\rangle = \lim_{T\to\infty}\frac{1}{T}\int_0^T \mathbf{F}_\phi(t)\;dt = 0
    \end{equation}
    Furthermore, the power dissipated by drag ${\mathcal W}_\phi(t):=\mathbf{F}_\phi(t)\bdot\bV(t)$ also has zero long-time 
    average:
    \begin{equation}
        \langle{\mathcal W}_\phi\rangle = \lim_{T\to\infty}\frac{1}{T}\int_0^T {\mathcal W}_\phi(t)\;dt = 0
    \end{equation}
\end{proposition}
\begin{proof}
    The total force is known to be given also by $\mathbf{F}_\phi = -d\mathbf{I}_\phi/dt$ (see e.g. \cite{batchelor1967fluid, lighthill1986informal}), the time derivative of an impulse 
    \begin{equation}
        \mathbf{I}_\phi(t) = -\int_{\partial B}\phi(t)\bn\;dS
    \end{equation}
    Since $\bV$ is bounded in time, $\|\phi(t)\|_{L^\infty(\partial\Omega)} < C$ for some constant $C$ and for all $t\ge0$. Thus, the 
    impulse $\mathbf{I}_\phi$ is also globally bounded in time and the long-time average of $\mathbf{F}_\phi$ must vanish. To see that 
    the same is true for the expended power ${\mathcal W}_\phi(t),$ we use the fact that $\mathbf{I}_\phi(t)={\mathbb M}_A\bV(t)$
    where ${\mathbb M}_A$ is a $3\times 3$ positive-definite matrix, known as the ``added mass tensor'', which depends 
    only on the set $B$ and not on time $t$ (see again \cite{batchelor1967fluid, lighthill1986informal}). In that case, 
    ${\mathcal W}_\phi(t)=\frac{d}{dt}\left( \frac{1}{2}\bV(t)^\top {\mathbb M}_A\bV(t)\right)$ and a similar 
    argument applies. 
\end{proof}

\begin{figure}
    \centering
    \includegraphics[width=0.75\linewidth]{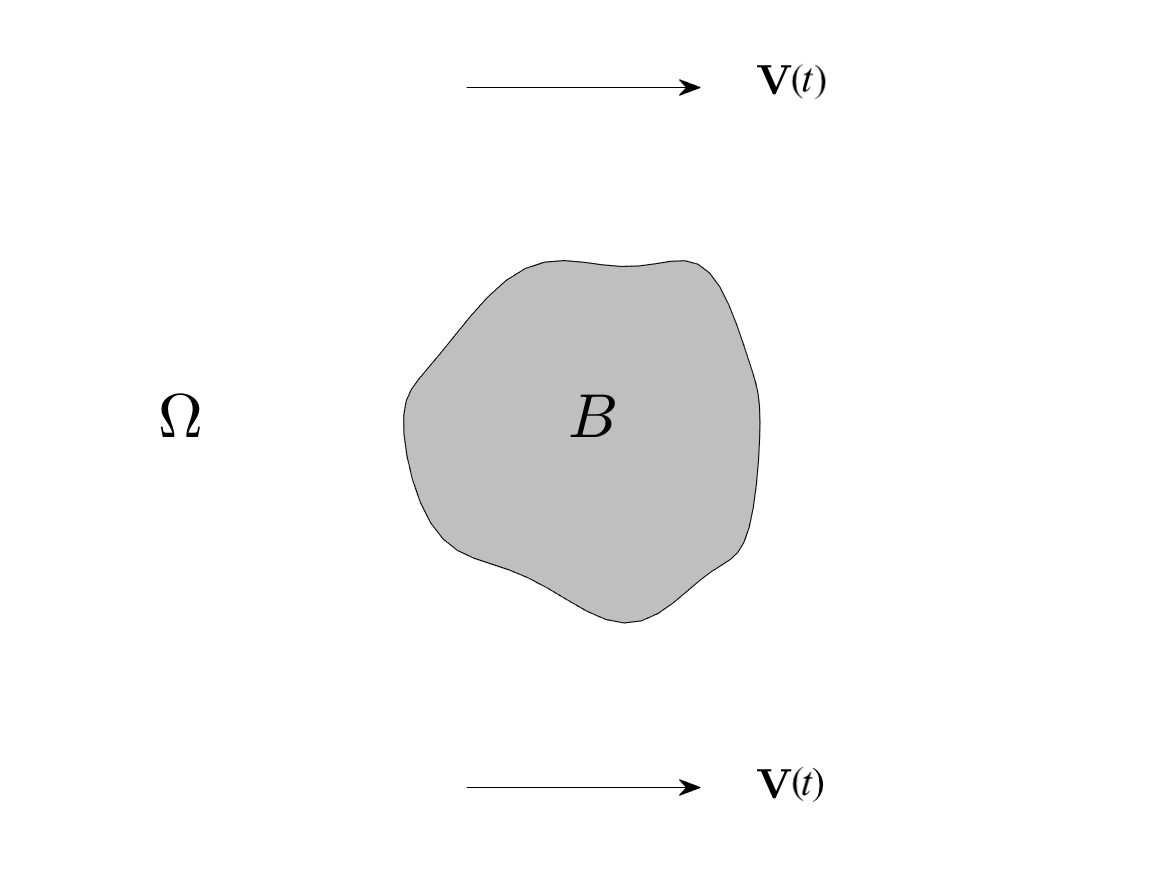}
    \caption{Flow around a finite body $B$ in an unbounded region $\Omega$ filled with an incompressible fluid moving at a velocity $\bV(t)$ at far distances.}
    \label{fig:external-flow}
\end{figure}
This result seems inconsistent with typical non-vanishing drag observed in laboratory experiments (e.g. \cite{achenbach1972experiments}) at high Reynolds numbers. As a possible resolution, we follow the approach in our previous work \cite{quan2022inertial, quan2024onsager} and study viscosity solutions of Euler equations.

\subsubsection{Prior work}\label{subsec:prior-work}
Our result relies on our previous work \cite{quan2022inertial}, \cite{quan2024onsager}, which establishes the validity of Josephson-Anderson relation for weak Euler solutions obtained in the zero-viscosity limit. 
This relation, derived in \cite{eyink2021josephson, eyink2021jaerratum} for flow past a smooth body at finite Reynolds number, equates the power dissipated by rotational fluid motions with the flux of vorticity across the flow lines of the potential Euler solution. The derivation starts with an assumed strong Navier Stokes solution $\bu^{\nu}$ satisfying
\begin{align}
    \partial_t\bu^\nu + \grad\bdot(\bu^\nu\otimes\bu^\nu) + \grad p^\nu &= - \nu\grad\times\bomega^\nu,\quad  \grad\bdot\bu^\nu = 0 &\text{on }\Omega\times(0,T)\label{eq:NS}\\
    \bu^\nu &= \bzed &\text{on } \partial\Omega\times(0,T)\label{eq:no-slip}
\end{align}
and \eqref{eq:far-distance-vel} at far distance.
The next step of the derivation involves decomposing $\bu^{\nu}$ into the background potential flow solution $\bu_\phi = \nabla\phi$ and a solenoidal field $\bu_\omega^{\nu}$, which corresponds to the rotational wake behind the body,
as follows:
\be \bu^\nu=\bu_\phi+\bu^\nu_\omega. \lb{decomp} \ee
The field $\bu_\omega^\nu$ satisfies the following equation that expresses local conservation of vortex momentum
\begin{equation}
    \partial_t\bu_\omega^\nu
    + \grad\bdot (\bu_\omega^\nu\otimes\bu_\omega^\nu+\bu_\omega^\nu\otimes\bu_\phi+\bu_\phi\otimes\bu_\omega^\nu) 
    +\grad p_\omega^\nu 
    =- \nu\grad\btimes\bomega^\nu, \lb{NS-omega-mom2} 
\end{equation}
subject to the boundary condition $\bu_\omega^{\nu} = -\bu_\phi$ on $\partial B$ and the initial condition $\bu_\omega^{\nu}(0) = \bu^{\nu}_0-\bu_\phi(0).$ The pressure $p_{\omega}^{\nu}$ is to be determined by the divergence-free constraint $\grad\cdot\bu_{\omega}^{\nu} = 0$. In that case, the total drag force on the body is a 
sum $\mathbf{F}^\nu(t)=\mathbf{F}_\phi(t)+\mathbf{F}^\nu_\omega(t),$ consisting of a potential part given by \eqref{eq:total-force} 
and a rotational part given by
\begin{equation}
    \mathbf{F}_\omega^\nu(t) \coloneqq \int_{\partial B}[-p_\omega\bn + 2\nu\bS\bdot\bn]\;dS. 
\end{equation}
The Josephson-Anderson relation states that the power transmitted to rotational motions 
${\mathcal W}_\omega^\nu(t):=\mathbf{F}_\omega^\nu(t)\bdot\bV(t)$ is given instantaneously by 
\be {\mathcal W}_\omega^\nu(t)= -\int_\Omega \mathbf{u}_\phi\bdot(\mathbf{u}_\omega^\nu\btimes\bomega^\nu
-\nu\grad\btimes\bomega^\nu)\, \, dV, \ee 
so that the total power expended is given by 
${\mathcal W}^\nu(t)= {\mathcal W}_\phi(t) + {\mathcal W}_\omega^\nu(t)$.

As in \cite{quan2024onsager}, we assume that the vortex momentum equations admit 
strong solutions $\bu^{\nu}_\omega$ for arbitrarily large Reynolds numbers. 
We also assume that for $q\geq 2,$
$(\mathbf{u}^{\nu}_\omega)_{\nu>0}$ converges strongly to $\mathbf{u}_\omega$ in $L^q((0,T),L_{\text{loc}}^q(\bar{\Omega}))$:
    \be
        \mathbf{u}^{\nu}_\omega\xrightarrow[L^q((0,T),L_{\text{loc}}^q(\bar{\Omega}))]{\nu\to0}\mathbf{u}_\omega. \label{u-omega-L2Conv}
    \ee    
and for $q\geq 1$ that $(p^{\nu}_\omega)_{\nu>0}$ converges strongly to $p_\omega$ in $L^q((0,T),L_{\text{loc}}^q(\bar{\Omega}))$. The notion of convergence in $L_{\text{loc}}^p(\bar\Omega)$ is essentially convergence in $L^p$ locally in the interior of $\Omega$, plus uniform boundedness in a neighborhood of $\partial\Omega$. 
See Section \ref{sec:proof-thm-1} for the precise definition.
In this case, the limit solves
the inviscid version of Eq.\eqref{NS-omega-mom2}:
\begin{equation}
    \partial_t\bu_\omega
    + \grad\bdot (\bu_\omega\otimes\bu_\omega+\bu_\omega\otimes\bu_\phi
    +\bu_\phi\otimes\bu_\omega) 
    + \grad p_\omega = 0, \quad\grad\bdot\bu_\omega=0 \label{E-omega-mom2} 
\end{equation}
subject to initial value $\bu_\omega(\cdot,0) \equiv \lim_{\nu\to}\bu_\omega^\nu(\cdot, 0)$, \textit{in the sense of distribution}:
\beal
    \int_0^{T}\int_{\Omega}\partial_t\barphi\bdot\bu_\omega 
    &+ \int_{\Omega} \bu_\omega(\bx, 0)\bdot\barphi(\bx, 0)\, dV \\
    &=-\int_0^T\int_\Omega\grad\barphi :(\bu_\omega\otimes\bu_\omega+\bu_\omega\otimes\bu_\phi
    +\bu_\phi\otimes\bu_\omega) \, dV dt
    \label{eq:u-omega-weak-form}
\eeal
for every $\barphi\in C_c^{\infty}(\Omega\times[0,T))$ with $\grad\bdot\barphi = 0$. 

The kinetic energy of the rotational motions is expected to be globally finite uniformly in Reynolds number, 
i.e. $\bu_\omega^\nu\in L^2(0,T;L^2(\Omega))$ uniformly in $\nu.$ An asymptotic multipole 
expansion shows  that $\bu_\omega^\nu$ is a dipole to leading order 
and decays as $|\bu_\omega^\nu|=O(r^{-3})$ for $r=|\bx|\to\infty$ \cite{eyink2021josephson,eyink2021jaerratum}. 
This decay can expected to remain true as $\nu\to 0$ because the dipole moment is the fluid impulse
$\bI_\omega^\nu(t),$ which should have an inviscid limit $\bI_\omega(t)$ whose time-derivative is 
${\bf F}_\omega(t)$. A basic assumption of \cite{quan2024onsager}, strengthening \eqref{u-omega-L2Conv},  
was that
    \be
        \mathbf{u}^{\nu}_\omega\xrightarrow[L^2((0,T),L^2(\Omega)]{\nu\to0}\mathbf{u}_\omega, \label{u-omega-L2Conv-glob}
    \ee    
a condition required for the rigorous derivation of the Josephson-Anderson relation in the inviscid limit. 
Formulation of similar assumptions for the rotational pressure $p_\omega^\nu$ requires more careful discussion. 
Despite representing rotational motions, nevertheless $\bu_\omega\sim \grad \phi_\omega$ as $r\to\infty$
because the asymptotic dipole field is potential.
In that case, the pressure $p_\omega$ is expected to be given asymptotically by the Bernoulli relation and 
the leading-order contribution is $p_\omega\sim -\partial_t\phi_\omega=O\left(r^{-2}\right)$ as $r\to\infty.$
See \cite{wu1981theory,lighthill1986informal,eyink2021jaerratum}. 
The pressure of the rotational flow in the inviscid limit can thus be expected to satisfy 
$p_\omega\in L^q(0,T;L^q(\Omega))$ only for $q>3/2.$ These physical expectations are 
incorporated into the definitions and theorem statements below. 

In particular, we shall say that $\bu_\omega$ is a finite-energy weak solution of the ideal vortex-momentum equation \eqref{E-omega-mom2} if $\bu_\omega\in L^2(0,T;H(\Omega))$ and satisfies \eqref{eq:u-omega-weak-form}. Here $H(\Omega)$ is the function space of solenoidal vector fields, defined by the $L^2$ completion of the space
    $\left\{\bv\in C_c^{\infty}(\Omega): \grad\bdot \bv = 0\right\}$
such that any vector field in $H(\Omega)$ satisfies both the divergence-free condition in the distributional sense and the no-flow through condition in the trace sense in $H^{-1/2}(\partial\Omega);$ see e.g. Theorem III.2.3. in \cite{galdi2011introduction}. We summarize these properties in the following equivalent formulation of $H(\Omega)$:
\begin{equation}
    H(\Omega)= \left\{\bv\in L^2(\Omega): \grad\bdot \bv = 0\mbox { in }\Omega,\;\; \bv\bdot\bn = 0\mbox { on }\partial\Omega\right\}\label{eq:characterization-H-Omega-2}
\end{equation}
Another convenient characterization of $H(\Omega)$ (see Chap. III of \cite{galdi2011introduction} or Section 1.6 of \cite{tsai2018lectures}) is 
\begin{equation}
    H(\Omega) =\left\{\bv\in L^2(\Omega): \int_{\Omega}\bv\bdot\grad \psi \, dV  = 0, \;\;\;
    \forall \psi \in W_{\text{loc}}^{1,2}(\Omega) \text{ s.t. } \grad \psi\in L^2(\Omega)\right\}\label{eq:characterization-H-Omega-1}
\end{equation}
These definitions of $H(\Omega)$ require $\Omega$ to be only a locally Lipschitz domain, which can bounded or unbounded. 
One of the new conclusions of the present work is that, under the assumption \eqref{u-omega-L2Conv-glob},
the inviscid limit $\bu_\omega$ is a finite-energy weak solution of the ideal equations \eqref{E-omega-mom2}, in the sense discussed above. 

An advantage of the decomposition \eqref{decomp} is that one obtains 
as easy corollaries corresponding results for the inviscid limit 
of the full fields $\bu^\nu,$ $p^\nu.$
The assumptions \eqref{u-omega-L2Conv} imply that for $q\geq 2$, $(\mathbf{u}^{\nu})_{\nu>0}$ converges strongly to $\mathbf{u}$ in $L^q((0,T),L_{\text{loc}}^q(\bar{\Omega}))$:
\be
    \mathbf{u}^{\nu}\xrightarrow[L^q((0,T),L_{\text{loc}}^q(\bar{\Omega}))]{\nu\to0}\mathbf{u}. \label{L2Conv}
\ee    
and for $q\geq 1$ that $(p^{\nu})_{\nu>0}$ converges strongly to $p$ in $L^q((0,T),L_{\text{loc}}^q(\bar{\Omega}))$. Furthermore, the inviscid limits 
$\bu$ solve the Euler equation in the sense of distributions:
\beal
    \int_0^{T}\int_{\Omega}\partial_t\barphi\bdot\bu 
    + \grad\barphi :(\bu\otimes\bu) \, dV dt
    + \int_{\Omega} \bu(\bx, 0)\bdot\barphi(\bx, 0)\, dV
    = 0
    \label{eq:u-weak-form}
\eeal
for divergence-free $\barphi.$ Under the stronger 
assumption \eqref{u-omega-L2Conv-glob}, 
$\bu\in L^2(0,T; H_{\text{loc}}(\Omega))$, where
\begin{equation}
    H_{\text{loc}}(\Omega)\coloneqq \{\mathbf{v}\in L^2_{\text{loc}}(\Omega): \grad\bdot \mathbf{v} = 0\mbox{ in }\Omega, \mathbf{v}\bdot\bn = 0\mbox{ on }\partial\Omega\}
\end{equation}
These statements parallel those obtained for weak solutions in bounded domains.

Invoking the assumptions \eqref{u-omega-L2Conv}, Theorem 1 of \cite{quan2022inertial} showed that the distributional limit of the wall shear stress $\btau_w^{\nu} = 2\nu\bS^{\nu}\bn = \nu\bomega\times\bn$ at the boundary exists for $\nu\to0$:
\begin{equation}
    \btau^{\nu}_w \xrightarrow{\nu\to0} \btau_w \text{ in } D'((\partial B)_T,\mathcal{T}(\partial B)_T) \lb{tau-lim}
\end{equation}
This result was only established in \cite{quan2022inertial} for the open time 
interval $(0,T),$ but we extend that result here to include the initial data.
More precisely, we will show that this limit exists as a distributional section of the tangent bundle of the space-time manifold $(\partial B)_T=\partial B\times [0,T)$, where we 
assume, as in \cite{quan2022inertial}, that $B\subset {\mathbb R}^3$ is closed, bounded, and has boundary $\partial B=\partial \Omega$, which is a $C^\infty$ manifold 
embedded in ${\mathbb R}^3.$ Of course, $\partial(\partial B)=\emptyset,$ but now 
$(\partial B)_T$ has a boundary $\partial B\times\{0\}$. See Section 2 of \cite{quan2022inertial} for notations and conventions 
regarding distribution theory on manifolds.

Since $\bu_\phi\bdot\bn=0$ on $\partial B$, and since 
$\partial B$ is compact, it follows that we may interpret $\bu_\phi|_{\partial B}\in D((\partial B)_T,{\mathcal T}^*(\partial B)_T)$ as a smooth section of the cotangent bundle of $(\partial B)_T$.
Thus, the dot product with the distribution $\btau_w\in D'((\partial B)_T,{\mathcal T}(\partial B)_T)$ 
obtained by Theorem 1 of \cite{quan2022inertial} can be defined as $\bu_\phi\bdot\btau_w\in D'((\partial B)_T)$ 
by 
    \be \langle \bu_\phi\bdot\btau_w,\psi\rangle:=\langle\btau_w,\psi\bu_\phi|_{\partial B}\rangle, 
    \quad \forall \psi \in D((\partial B)_T)
\ee 
Furthermore, under assumptions \eqref{u-omega-L2Conv} strengthened by taking $q\geq 3$ for $\bu_\omega$ and $q\geq 3/2$ for $p_\omega$, we obtain that the inviscid limit of viscous dissipation $Q^{\nu} = \nu|\bomega^{\nu}|^2$ exists as a non-negative distribution, and a balance equation of kinetic energy in the rotational wake, $E_{\omega}(t)\coloneqq \frac{1}{2}\int_{\Omega}|\bu_\omega(\cdot, t)|^2\, dV,$ emerges in the inviscid limit. In order to discuss boundary terms in this energy balance, we introduce a non-standard space of test functions
\begin{equation}
    \bar{D}(\bar\Omega\times[0,T))\coloneqq\{\varphi=\phi|_{\bar\Omega\times[0,T)}: \phi\in C_c^{\infty}(\mathbb{R}^3\times\mathbb{R})\}
\end{equation}

In summary, we may state our first theorem, which extends and consolidates results from \cite{quan2022inertial, quan2024onsager}:
\begin{theorem}\label{thm:rot-energy}
Let $(\mathbf{u}^{\nu}_\omega, p^{\nu}_\omega)$ be strong solutions of Eq.\eqref{NS-omega-mom2} on $\Bar{\Omega}\times [0,T)$ for $\nu>0$. 
    Assume that for some $q\ge 3$,
    \be
        \mathbf{u}^{\nu}_{\omega}\xrightarrow[L^q(0,T;L_{\text{loc}}^q(\bar\Omega))]{\nu\to0}\mathbf{u}_\omega, \quad\quad
        \mathbf{u}^{\nu}_{\omega}(0)\xrightarrow[L_{\text{loc}}^2(\bar\Omega)]{\nu\to0}0
        \label{assumption:vel-l3-conv}
    \ee    
    and     
    \be
        p^{\nu}_{\omega}\xrightarrow[L^{\frac{q}{2}}(0,T;L_{\text{loc}}^{\frac{q}{2}}(\bar\Omega))]{\nu\to0}p_\omega\label{assumption:p-l32-conv}
    \ee
    Then the limit $(\bu_\omega,p_\omega)$ solves
    the inviscid vortex momentum equation \eqref{E-omega-mom2} in the sense of distributions
    and
    \be
    \btau^{\nu}_w \xrightarrow{\nu\to0} \btau_w \text{ in } D'((\partial B)_T,\mathcal{T}(\partial B)_T)\label{tau-lim-2}
    \ee
    Also, $Q^{\nu} = \nu|\bomega^{\nu}|^2$ converges for this subsequence to a non-negative linear functional $Q$ 
    on $\Bar{D}(\Bar{\Omega}\times [0,T))$, in the sense that $\forall\varphi\in \Bar{D}(\Bar{\Omega}\times [0,T))$,
        \begin{align}
            \lim_{\nu \to 0}  \int_0^T \int_{\Omega} \varphi\, Q^\nu \,  \, dV\, dt = \langle Q,\varphi\rangle
            \label{viscDissLimit2}
        \end{align}
    with $\langle Q,\varphi\rangle\geq 0$ for $\varphi\geq 0.$ Furthermore, an inviscid version of the balance equation for the rotational energy holds in the sense that
   for all $\varphi \in \Bar{D}(\Bar{\Omega}\times[0,T))$, $\psi = \varphi|_{\partial B}$,
    \begin{align}
    \nonumber
        &&-\int_{\Omega}\frac{1}{2}\varphi(\bx,  0)|\bu_{\omega}(0)|^2\, dV
        -\int_0^T\int_{\Omega}\frac{1}{2}\partial_t\varphi|\bu_{\omega}|^2 + 
        \grad\varphi\bdot
        \left[\frac{1}{2}|\bu_{\omega}|^2\bu+p_{\omega}\bu_{\omega}\right]\, dV dt\\
        &&\hspace{30pt} =\langle\bu_\phi\bdot \btau_w,\psi\rangle
        -\langle Q,\varphi\rangle  
        -\int_0^T\int_{\Omega}\varphi\grad\mathbf{u}_{\phi}\bdots
        \mathbf{u}_{\omega}\otimes\mathbf{u}_{\omega}\, dV dt\label{eq:rot-energy-loc}
    \end{align}
    Finally, if the convergence holds in global $L^2$
    \begin{equation}
        \bu_\omega^\nu \xrightarrow[L^2(0,T;L^2(\Omega))]{\nu\to0} \bu_\omega \label{assumption:vel-global-l2}
    \end{equation}
    then the inviscid limit $\bu_\omega$ belongs to $L^2(0,T;H(\Omega))$  and is a 
    finite-energy weak solution of \eqref{E-omega-mom2}. 
    In that case, the limiting power dissipated by drag from rotational motions, ${\mathcal W}_\omega(t)=
    \lim_{\nu\to 0} {\mathcal W}_\omega^\nu(t),$ exists and is given by an inviscid version 
    of the Josephson-Anderson relation: 
    \be {\mathcal W}_\omega(t) = -\int_\Omega \grad\bu_\phi(\cdot,t)\bdots \mathbf{u_\omega}(\cdot,t)\otimes \mathbf{u_\omega}(\cdot,t)\,
     \, dV + \int_{\partial\Omega} \mathbf{u}_\phi(\cdot,t)\bdot\btau_w(\bdot,t)\, dA\label{eq:ja-relation}
     \ee 
    which holds distributionally in time. 
\end{theorem}

\begin{remark}
    The energy $E_{\omega}$ is also called ``relative energy'' in the PDE literature. It compares weak and strong solutions of Euler equations, 
    providing a common method for proving weak-strong uniqueness results (see \cite{wiedemann2018weak} for an overview). In fact, the proof of our next main 
    Theorem \ref{thm:external} is an example of the relative energy method, 
    which relies on Theorem \ref{thm:rot-energy}. See Section \ref{sec:proof-thm-3}
\end{remark}

\subsubsection{Main theorem statement}
We now state our main result regarding weak-strong uniqueness in the context of flow past a solid body:
\begin{theorem}\label{thm:external}
    Let $(\bu_\omega, p_\omega)$ be the limiting weak solution in Theorem \ref{thm:rot-energy} with vanishing limiting initial value $\bu_\omega(0) = \bzed$. Assume further that 
    \begin{align}
        &\bu_{\omega} \in L^{2}(0,T; L^2(\Omega)) \cap L^3(0,T; L^3(\Omega))\label{assumption:ext-vel-global}\\
        &p_{\omega} \in L^{q}(0,T; L^{q}(\Omega))
        \mbox{ for some } q\in (\frac{3}{2},2)\label{assumption:ext-p-global}
    \end{align}
    If \red{$\langle\bu_{\phi}\bdot\btau_w(\cdot,t),1\rangle:= \int_{\partial\Omega} \mathbf{u}_\phi(\cdot,t)\bdot\btau_w(\bdot,t)\, dA\equiv 0$ distributionally in time}, then 
    \begin{equation}
        \bu_{\omega}(\bx,  t) = \bzed, \text{ for a.e. } (\bx,  t) \in \Omega\times(0,T) 
    \end{equation}
    In other words, viscosity solutions $\bu$ of the Euler equations have the weak-strong uniqueness
    property that $\bu(\bx,  t) = \bu_{\phi}(\bx,  t)$ for almost every $x, t\in \Omega\times(0,T)$ when \red{$\langle \bu_\phi\bdot\btau_w(\cdot,t),1\rangle\equiv 0$} holds. 
    Furthermore, limiting power dissipated by rotational motions vanishes, ${\mathcal W}_\omega(t)\equiv 0,$ and likewise
    anomalous dissipation vanishes, $Q\equiv 0.$
\end{theorem}

\begin{remark}
As mentioned earlier, our result aligns closely with the general \red{results of Bardos-Titi \cite{bardos2013mathematics} and Kelliher \cite{kelliher2017observations}.}
Theorem 4 \red{in \cite{bardos2013mathematics}} states that weak-strong uniqueness holds under the condition $\btau_w = \bzed$ for weak-* limits in $L^\infty(0, T; L^2(\Omega))$ of Navier-Stokes solutions $\bu^{\nu}$. Their result is more general in that it considers weak Leray-Hopf solutions and it does not require any \textit{a priori} assumption to obtain inviscid limits. However their result is also less general, as it assumes solutions of finite global energy. 
\red{Our condition $\langle\bu_{\phi}\bdot\btau_w(\cdot,t),1\rangle \equiv0$ 
a slight strengthening of Kelliher's for weak-strong uniqueness, proved in \cite{kelliher2017observations},  Theorem 8.1 for inviscid limits of 
weak Navier-Stokes solution in 2D bounded domains and and stated formally in 
\cite{kelliher2017observations}, Remark 8.3 
for the same situation in 3D. Note that under our assumption of strong Navier-Stokes solutions, Kelliher's Remark 8.3 follows as a rigorous theorem.} 
\end{remark}


\begin{remark}
The Drivas-Nguyen condition \eqref{cond:no-flow-through} of uniform continuity at the wall of the normal velocity was shown to imply that 
$\btau_w = \bzed$ (and hence $\bu_{\phi}\bdot\btau_w = 0$) by Theorem 3 in \cite{quan2022inertial}. This implies weak-strong uniqueness for viscosity solutions by Theorem 4.1(2) of \cite{bardos2013mathematics}. 
\end{remark}

In fact, the Drivas-Nguyen condition implies weak-strong uniqueness for general admissible Euler solutions in bounded domains, which is the statement of our next Theorem \ref{thm:internal}. To formulate it, we define the distance to the boundary $d(\bx)\coloneqq \inf_{y\in\partial \Omega}|\bx-\by|$ and the open tubular neighborhood $\Omega_\epsilon\coloneqq\{\bx\in\Omega: d(\bx)<\epsilon\}.$ It can be shown for some sufficiently 
small $\epsilon > 0$ that,  
for any $\bx\in\Omega_{\epsilon}$, there exists a unique point $\pi(\bx)\in\partial\Omega$ such that
\begin{equation}
    d(\bx) = |\bx - \pi(\bx)|,\quad \grad d(\bx) = \bn(\pi(\bx))
    :=\bn(\bx) \label{def:dist-func}
\end{equation}
where $\bn$ is the unit normal on $\partial B$ pointing into $\Omega$ and 
$\bn(\bx)$ smoothly extends $\bn$ into $\Omega_\epsilon.$
For example, see Section 14.6 in \cite{gilbarg1977elliptic}. 
We can now state: 
\begin{theorem}\label{thm:internal}
Let $\Omega\subset\mathbb{R}^n$ be a bounded, simply-connected domain with $C^{\infty}$ boundary for $n=2$ or $3$. 
Suppose that $\bU\in C^1(\bar{\Omega}\times[0,T])$ is a strong solution of incompressible Euler equations with $\bU(\cdot, 0) = \bu_0$, and
$\bu\in L^{\infty}(0,T; H(\Omega))$ is an admissible weak solution of Euler on $\Omega$ for which there exists some $\epsilon>0$ s.t.
\begin{align}
    \bu\in L^{\infty}(0,T; L^{\infty}(\Omega_{\epsilon}))\label{cond:vel-boundary-bound}
\end{align}
and for $0<\delta<\epsilon,$
\begin{align}
    \lim_{\delta\to0}\norm{\bn\cdot\bu}_{L^{\infty}(0,T; L^{\infty}(\Omega_{\delta}))} = 0\label{cond:no-flow-through}
\end{align}
Then $\bu = \bU$ for almost every $(\bx, t)\in\Omega\times(0,T)$.
\end{theorem}
\begin{remark}
    The assumption \eqref{cond:no-flow-through} can be viewed as a uniform continuity requirement on the wall-normal velocity at the boundary, weaker than the near-wall $C^\alpha$ condition on $\bu$ in \cite{bardos2014non} and the $C^0$ condition in \cite{wiedemann2018weak}. This assumption \eqref{cond:no-flow-through} is motivated by condition (11) used in \cite{drivas2018nguyen} to establish energy conservation for weak Euler solutions. The significance of such boundary behavior was noted in \cite{bardos2018onsager, bardos2019onsager}.
\end{remark}


\section{Proof of Theorem \ref{thm:rot-energy}}\label{sec:proof-thm-1}
In this section we prove Theorem \ref{thm:rot-energy}, primarily by incorporating initial data into Theorem 1 of \cite{quan2022inertial} and Theorem 4 of \cite{quan2024onsager}. Let $\delta > 0$ be a small positive number. We first define the manifold $\partial B_{\delta,T} \coloneqq \partial B \times (-\delta, T) \subset \mathbb{R}^3 \times \mathbb{R}$, endowed with the natural $C^{\infty}$ structure and without boundary. We then smoothly extend all the quantities considered in \cite{quan2024onsager} in time to $(-\delta, T)$ and truncate them by multiplying by a time cutoff function.

Specifically, consider the skin friction $\btau_w^{\nu}: C^{\infty}(\bar{\Omega}\times[0,T), \mathbb{R}^3)$, which can be identified as a smooth section of the tangent bundle of $\partial B\times[0,T)$ (see Section 2 and 3 in \cite{quan2022inertial}). Since $\partial B\times[0,T)$ is a closed subset of $\partial B_{\delta,T}$, we can extend $\btau_w^{\nu}$ to a smooth section of the 
tangent bundle of $\partial B_{\delta,T}$ (see Lemma 10.12 in \cite{lee2013smooth}). 
We denote the extended section as $\widehat{\btau_w^\nu}$ and it belongs to the space of smooth sections of the tangent bundle $D(\partial B_{\delta,T}, \mathcal{T}(\partial B_{\delta,T}))$, which is a Fréchet space equipped with the seminorms:
\begin{equation}
    p_{s,m,i}(\bpsi)\coloneqq \sum_{j=1}^k\Tilde{p}_{s,m,i}((\Pi_2\Phi_i)^j\circ\bpsi|_{V_i}\circ\phi_i^{-1}) \lb{seminorm}
\end{equation}
for $\bpsi\in D(\partial B_{\delta,T}, \mathcal{T}(\partial B_{\delta,T}))$. 
Here $\cup_{i\in I}(V_i, \Phi_i)$ is a smooth structure of the tangent bundle $\mathcal{T}(\partial B_{\delta,T})$ with open subsets $V_i\subset \partial B_{\delta,T}$ and $\Phi_i:\Pi^{-1}(V_i)\to\mathbb{R}^4\times\mathbb{R}^3$, where $\Pi$ is the natural projection map from $\mathcal{T}(\partial B_{\delta,T})$ to $\partial B_{\delta,T}$. 
Moreover, $\cup_{i\in I}(\phi_i, V_i)$ with $ \phi_i:V_i\to\mathbb{R}^4$  is a smooth structure on $\partial B_{\delta,T}$ that satisfies $\Pi_1\phi_i = \phi_i\Pi$. 
Here, $\Pi_1$ projects onto the first factor of $\mathbb{R}^4\times\mathbb{R}^3$ and $\Pi_2$ on the second.
Lastly, $\{\Tilde{p}_{s,m,i}\}$ in \eqref{seminorm} is a countable and separating basis of seminorms on $C^{\infty}(\phi_i(V_i))$, defined as
\begin{equation}
    \Tilde{p}_{s,m,i}(f) \coloneqq \sup_{x\in K_m^{(i)}, |\alpha| \le s}|D^{\alpha}f(x)|
\end{equation}
for $f\in C^{\infty}(\phi_i(V_i))$. For more details on the setup of the smooth section space, see Section 2.2 in \cite{quan2022inertial}.

We first detail how to refine Theorem 1 in \cite{quan2022inertial}. 
The extended smooth section $\widehat{\btau_w^\nu}$ can be canonically identified as a distributional section in $D'(\partial B_{\delta,T}, \mathcal{T}(\partial B_{\delta,T}))$. We further truncate it in time by multiplying with a characteristic function $\chi_{[0,T)}$ to obtain $\widetilde{\btau_w^\nu}\coloneqq\chi_{[0,T)}\widehat{\btau_w^\nu}$, which is still a distributional section in $D'(\partial B_{\delta,T}, \mathcal{T}(\partial B_{\delta,T}))$.
The extension operator for smooth sections of the tangent bundle is defined exactly as in \cite{quan2022inertial}, mapping $\ext: D(\partial B_{\delta,T}, \mathcal{T}(\partial B_{\delta,T})) \to \bar{D}(\bar{\Omega} \times (-\delta, T), \mathbb{R}^3)$. Then, for any $\bpsi \in D(\partial B_{\delta,T}, \mathcal{T}(\partial B_{\delta,T}))$, we have $\barphi = \ext(\bpsi) \in \bar{D}(\bar{\Omega} \times (-\delta, T), \mathbb{R}^3)$. We can deduce by the incompressible Navier–Stokes equations and integration by parts that 
\begin{equation}
    \begin{aligned}\label{basicT}
        \innerprod{\widetilde{\btau_w^{\nu}}}{\bpsi}
        &= \int_{-\delta}^T\int_{\partial B}\widetilde{\btau_w^{\nu}}\bdot\bpsi dS dt
        = \int_{0}^T\int_{\partial B}\btau_w^{\nu}\bdot\barphi|_{\partial B\times[0,T)} dS dt\\
        &= \int_{\Omega}\barphi(\bx, 0)\bdot\bu_0^{\nu} \, dV
        + \int_0^T\int_{\Omega}\partial_t\barphi\bdot \mathbf{u}^{\nu}
        +\grad\barphi\bdots[\mathbf{u}^{\nu}\otimes\mathbf{u}^{\nu} 
        + p^{\nu}\mathbf{I}] \, dV dt\\
        &+\int_0^T \int_{\Omega} \nu \triangle \barphi \bdot\bu^{\nu} \, dV dt
    \end{aligned}
\end{equation}
In order to study boundary effects in the zero-viscosity limit, we need a notion of function and convergence in an unbounded domain that considers both the interior and neighborhoods of the boundary: 
\begin{definition}\label{def:extended-lp}
    For any $p\in[1,\infty]$, a function $f\in L^p_{\text{loc}}(\Omega)$ on an open set $\Omega\in\mathbb{R}^3$ (possibly unbounded) is said to be \textit{locally $L^p$ up to the boundary} if $\norm{f}_{L^p(\Omega_{\epsilon})} < \infty$ for some $\epsilon>0$. We denote the space of such functions as
    \begin{equation}
        L^p_{\text{loc}}(\bar{\Omega})\coloneqq \{f\in L_{\text{loc}}^p(\Omega); \|f\|_{L^p(\Omega_{\epsilon})}<\infty \text{ for some }\epsilon>0\}
    \end{equation}
\end{definition}
\noindent
This definition is independent of the choice of $\epsilon$ as it implies that $\norm{f}_{L^p(\Omega_{\delta})} < \infty$ for all $\delta > 0$. 
It is easy to show that $L_{\text{loc}}^p(\bar{\Omega})$ is equivalent to
\begin{equation}
    L^p_{\text{loc}}(\bar{\Omega})= \{f\in L_{\text{loc}}^p(\mathbb{R}^3):\, \|f\|_{L^p(\Omega\cap B)}<\infty \text{ for any open ball } B\subset\mathbb{R}^3 \}. 
\end{equation}
See, e.g., \cite{borchers1990equations}.

The next corollary is proved just as Lemma 1 in \cite{quan2024onsager}:
\begin{corollary}\label{cor:conv-int-boundary}
    For a sequence of functions $\{f_n\}_{n>0}$ in $L_{\text{loc}}^p(\bar\Omega)$, if 
    \beal
        &f_n\to f \quad \text{in } L^p_{\text{loc}}(\Omega)\\
        &f_n \text{ is uniformly bounded } \quad \text{in } L^p(\Omega_{\epsilon}) \text{ for some }\epsilon>0\label{uniform-bound-boundary}
    \eeal
    then $f\in L^p(\Omega_\epsilon)$ and thus $f\in L_{\text{loc}}^p(\bar\Omega)$.
\end{corollary}
\noindent
Then we can define convergence on $L_{\text{loc}}^p(\bar\Omega)$ as follows:
\begin{definition}\label{def:extended-lp-space-time}
    We say that a sequence of functions $f_n$ \textit{converges to} $f$ in $L_{\text{loc}}^p(\bar\Omega)$ if $f_n$ satisfies \eqref{uniform-bound-boundary}.
\end{definition}
\noindent
We extend Definition \ref{def:extended-lp} to functions varying in time as follows
\beal
    L^q(0,T; L^p_{\text{loc}}(\bar{\Omega}))\coloneqq 
    \{&f\in L^q(0,T; L_{\text{loc}}^p(\Omega)); \\
    &\|f\|_{L^q(0,T; L^p(\Omega_{\epsilon}))}<\infty \text{ for some }\epsilon>0\}
\eeal
for $p,q\in[1,\infty]$. A result similar to Corollary \ref{cor:conv-int-boundary} applies to functions in Definition \ref{def:extended-lp-space-time}, allowing us to define convergence of $f_n$ to $f$ in $L^q(0,T; L^p_{\text{loc}}(\bar{\Omega}))$ if
\beal
    &f_n\to f \quad \text{in } L^q(0,T; L^p_{\text{loc}}(\Omega))\\
    &f_n \text{ uniformly bounded } \quad \text{in } L^q(0,T; L^p(\Omega_{\epsilon})) \text{ for some }\epsilon>0\label{uniform-bound-boundary-space-time}
\eeal

Note that due to the time truncation, the integration in time is essentially performed over $(0, T)$, which leads to the same local Navier–Stokes equations integrated against test functions as in \cite{quan2022inertial}, but with an additional term involving the initial data $\bu_0^{\nu}$. Given $\bu_0^\nu = \bu_\phi(0) + \bu_\omega^\nu(0)$, it follows from \eqref{assumption:vel-l3-conv} that
\begin{equation}
    \bu_0^{\nu} \to \bu_0 \text{ in } L_{\text{loc}}^2(\bar{\Omega})\label{assumption:initial_data}
\end{equation}
An easy argument similar to that in Lemma 1 of \cite{quan2022inertial} shows that for all $\barphi\in\bar{D}(\bar\Omega\times(-\delta, T), \mathbb{R}^3)$
\begin{equation}
    \lim_{\nu\to0}\int_{\Omega}\barphi(\bx, 0)\cdot \bu_0^{\nu}(\bx) \, dV 
        = \int_{\Omega}\barphi(\bx, 0)\cdot \bu_0(\bx) \, dV 
\end{equation}
Furthermore, we have
\begin{equation}
    \left|\int_{\Omega}\barphi\cdot \bu_0 \, dV \right|
    \lesssim \norm{\bu_0}_{L^2(\text{supp}(\barphi(0)))}\sup_{i\in I}p_{1,m,i}(\bpsi)
\end{equation}
Other terms in \eqref{basicT} can be treated in the same way as in \cite{quan2022inertial}. Thus, we can conclude that
\red{
\begin{equation}
    \widetilde{\btau^{\nu}_w} \xrightarrow{\nu\to0} \btau_w 
    \text{ in } D'(\partial B_{\delta,T}, \mathcal{T}(\partial B_{\delta,T})) \label{tau-lim-ext}
\end{equation}}
The limiting distribution \red{$\btau_w$} is clearly supported in $\partial B \times [0, T)$ and is independent of the extension to $\partial B_{\delta,T}$ by the limit of \eqref{basicT}. Therefore, we can interpret \red{$\btau_w$} as acting on smooth sections in $D(\partial B_{\delta,T}, \mathcal{T}(\partial B_{\delta,T}))$ restricted to $\partial B \times [0, T)$. In this way, we have justified the convergence of skin friction when smeared with test functions $\barphi$
such that $\barphi(\cdot,0)\neq \bzed.$

Next we discuss how to extend Theorem 4 in \cite{quan2024onsager}. 
Given that $\bu_\phi$ is tangent to $\partial B$, it can also be identified as a smooth section of the tangent bundle of $\partial B\times[0,T)$ and can be extended to a smooth section in $D(\partial B_{\delta,T}, \mathcal{T}(\partial B_{\delta,T}))$ \cite{lee2013smooth}, which we denote as $\hat{\bu}_\phi$. Since for any scalar test function in $\psi\in D(\partial B_{\delta,T})$ it follows that $\psi \hat{\bu}_{\phi} \in D(\partial B_{\delta,T}, \mathcal{T}(\partial B_{\delta,T}))$, 
the dot product with the distributional section \red{$\btau_w$} can be defined in the same way as in \cite{quan2024onsager}. We have \red{$\bu_\phi \bdot \btau_w \in D'(\partial B_{\delta,T})$} by setting
\red{
\begin{equation}
    \innerprod{\bu_{\phi}\bdot\btau_w}{\psi}
    \coloneqq \innerprod{\btau_w}{\psi\hat{\bu}_{\phi}|_{\partial B}},\quad
    \forall \psi\in D(\partial B_{\delta,T})
\end{equation}}
Because \red{$\bu_{\phi} \bdot \btau_w$} is also supported on $\partial B \times [0, T)$, we can define the distributional pairing with $\psi \in D(\partial B_{\delta,T})$ restricted to $\partial B \times [0, T)$, i.e. $\bar{D}(\partial B\times[0,T))$.

Following the proof of Theorem 4 in \cite{quan2024onsager}, we take an arbitrary $\varphi \in \bar{D}(\bar{\Omega} \times [0, T))$ with $\psi = \varphi|{\partial B} \in C^{\infty}(\partial B \times [0, T))$, and test the rotational energy equation against it. The local energy balance for the rotational flow now incorporates the initial data:
\begin{equation}
    \begin{aligned}
        -\int_{\Omega}&\frac{1}{2}\varphi(\bx, 0)|\bu_{\omega}^{\nu}(0)|^2 \, dV
        -\int_0^T\int_{\Omega}\frac{1}{2}\partial_t\varphi|\mathbf{u}_{\omega}^{\nu}|^2 + 
        \grad\varphi\bdot\left[\frac{1}{2}|\mathbf{u}_{\omega}^{\nu}|^2\mathbf{u}^{\nu}
        +p_{\omega}^{\nu}\mathbf{u}_{\omega}^{\nu}\right]\, dV dt\\
        &=\int_0^T\int_{\Omega}\varphi\grad\bdot(\nu\mathbf{u}_{\omega}^{\nu}\btimes\bomega)
        \, dV dt
        -\int_0^T\int_{\Omega}\nu\varphi|\bomega^{\nu}|^2 \, dV dt\\
        &-\int_0^T\int_{\Omega}\varphi\grad\mathbf{u}_{\phi}\bdots\mathbf{u}_{\omega}^{\nu}\otimes\mathbf{u}_{\omega}^{\nu} \, dV dt
    \end{aligned}\label{rotationTest}
\end{equation}
Under condition \eqref{assumption:vel-l3-conv}, it is easy to show using an argument similar to that in Lemma 1 of \cite{quan2024onsager} that
\begin{equation}
    \lim_{\nu\to0}\int_{\Omega}\frac{1}{2}\varphi(\bx, 0)|\bu_{\omega}^{\nu}(0)|^2 \, dV
    = \int_{\Omega}\frac{1}{2}\varphi(\bx, 0)|\bu_{\omega}(0)|^2 \, dV
\end{equation}
Integration by parts gives
\begin{equation}
    \begin{aligned}
        &\int_0^T\int_{\Omega}\varphi\grad\bdot(\nu\mathbf{u}_{\omega}^{\nu}\btimes\bomega) \, dV dt
        = \int_0^T\int_{\partial B}\psi\bu_{\phi}\bdot\btau_w^{\nu}dS dt \\
        &\hspace{30pt} - \int_0^T\int_{\Omega}\nu\Delta(\varphi\bu_{\phi})\cdot\bu^{\nu} \, dV dt
        + \int_0^T\int_{\Omega}\varphi\grad\bdot(\nu\bu^\nu\times\bomega^{\nu})\, dVdt
\end{aligned}\label{nu_u_omega}
\end{equation}
Extending $\psi$, $\bu_{\phi}$, $\btau_w^{\nu}$ in any way guaranteed by Lemma 2.26 and Lemma 10.12 in \cite{lee2013smooth} respectively for scalar functions and smooth sections, we have $\hat{\psi}\in D(\partial B_{\delta,T})$, $\hat{\psi}\hat{\bu}_{\phi}\in D(\partial B_{\delta,T}, \mathcal{T}(\partial B_{\delta,T}))$ and
\begin{equation}
    \int_0^T\int_{\partial B}\psi\bu_{\phi}\bdot\btau_w^{\nu}dS dt 
    = \int_{-\delta}^T\int_{\partial B}\hat{\psi}\hat{\bu}_{\phi}\bdot\widetilde{\btau_w^{\nu}}dS dt
    = \innerprod{\widetilde{\btau_w^{\nu}}}{\hat{\psi}\hat{\bu}_{\phi}} \label{sectional_prod}
\end{equation}
Thus as $\nu\to0$, \eqref{tau-lim-ext} gives \red{
\begin{equation}
    \int_0^T\int_{\partial B}\psi\bu_{\phi}\bdot\btau_w^{\nu}dS dt 
    = \innerprod{\widetilde{\btau_w^{\nu}}}{\hat{\psi}\hat{\bu}_{\phi}} 
    \to \innerprod{\btau_w}{\hat{\psi}\hat{\bu}_{\phi}}
    =  \innerprod{\bu_{\phi}\bdot\btau_w}{\hat{\psi}}\label{phi_tau_limit}
\end{equation}}
Due to the time cutoff by the characteristic function $\chi_{[0,T)}$ in $\widetilde{\btau_w^{\nu}}$, \eqref{sectional_prod} is invariant under choice of smooth extension and thus the limiting distributional product \eqref{phi_tau_limit} is likewise independent. Since \red{$\btau_w$} and \red{$\bu_{\phi} \bdot \btau_w$} are supported on $\partial B \times [0, T)$, we can then define for any $\psi \in C^{\infty}(\partial B \times [0, T))$ that
\red{\begin{equation}
    \innerprod{\bu_{\phi}\bdot\btau_w}{\psi}
    \coloneqq\innerprod{\btau_w}{\hat{\psi}\hat{\bu}_{\phi}|_{\partial B}}
\end{equation}}
All the other terms converge in the same way as in \cite{quan2024onsager} as $\nu\to0$ 
and we obtain the local energy balance in the inviscid rotational flow \eqref{eq:rot-energy-loc}.

Finally, we discuss the proofs under the strengthened global hypothesis 
\eqref{assumption:vel-global-l2}. The inviscid Josephson-Anderson relation
\eqref{eq:ja-relation} was already established under this assumption in 
Theorem 1 of \cite{quan2024onsager}. 
Since $\grad\bdot\bu_\omega^{\nu} = 0$ and $\bu_\omega^\nu\bdot\bn = -\bu_\phi\bdot\bn = 0$, the following equality holds
\begin{equation}
    \int_\Omega\bu_\omega^\nu(\bx, t)\bdot\grad v(\bx) \, dV = 0 
\end{equation}
for every $t\in[0,T]$ and every $v\in W_{\text{loc}}^{1,2}(\Omega)$ with $\grad v\in L^2(\Omega)$.
Then the global convergence \eqref{assumption:vel-global-l2} implies that 
\begin{equation}
    \int_\Omega\bu_\omega(\bx, t)\bdot\grad v(\bx) \, dV = 0 
\end{equation}
for a.e. $t\in[0,T]$. Thus, by the characterization \eqref{eq:characterization-H-Omega-1}, $\bu_\omega\in L^2(0,T; H(\Omega))$ 
and $\bu_\omega$ is a finite-energy weak solution of \eqref{E-omega-mom2}.

\section{Proof of Theorem \ref{thm:external}}\label{sec:proof-thm-2} 
We use a relative energy argument as in \cite{wiedemann2018weak}. Since $\bu_{\omega} \in L^{2}(0,T; H(\Omega))$, we can define for almost every $t\in(0,T)$,
\begin{equation}
    E_{\omega}(t) \coloneqq \int_{\Omega} |\bu_{\omega}(\bx,  t)|^2\, dV
\end{equation}
We start with the inviscid local energy balance in Theorem \ref{thm:rot-energy} 
\beal
    &
    -\int_{\Omega}\frac{1}{2}\varphi(\bx,  0)|\bu_{\omega}(\bx, 0)|^2\, dV
    -\int_0^T\int_{\Omega}\frac{1}{2}\partial_t\varphi|\bu_{\omega}|^2 + 
    \grad\varphi\bdot
    \left[\frac{1}{2}|\bu_{\omega}|^2\bu+p_{\omega}\bu_{\omega}\right]\, dV dt\\
    &=\langle\bu_\phi\bdot \btau_w,\psi\rangle
    -\langle Q,\varphi\rangle  
    -\int_0^T\int_{\Omega}\varphi\grad\mathbf{u}_{\phi}\bdots
    \mathbf{u}_{\omega}\otimes\mathbf{u}_{\omega}\, dV dt\label{fgRotationalLimit}
\eeal
for test functions $\varphi\in\bar{D}(\bar{\Omega}\times[0,T))$ with $\psi = \varphi|_{\partial B}$. 
Let $B_R = B(0,R)$ for some $R > 0$.
Then we define a specific test function
\begin{equation}
    \varphi = \chi_{(-\delta,\tau]}^{\epsilon}\chi_{B_R}^{\eta}\in C^{\infty}_c(\mathbb{R}^3\times \mathbb{R})
\end{equation}
as a product of two mollified characteristic functions respectively in time and space, with sufficiently small $\epsilon, \eta > 0$ and some fixed numbers $\tau\in(0,T)$ and $\delta > \epsilon > 0$. 
More specifically, we  define the mollified time characteristic function
\begin{equation}
    \chi_{(-\delta,\tau]}^{\epsilon} = \chi_{(-\delta,\tau]} * G_{\epsilon}    
\end{equation}
where $G$ is a standard smooth kernel in $C_c^{\infty}(\mathbb{R})$ such that $\text{supp}(G) \subset [-1,1]$, $G\ge 0$, and $\int_{\mathbb{R}} G(t)dt = 1$, and $G_{\epsilon}(t) = \frac{1}{\epsilon} G(\frac{t}{\epsilon})$.
We mollify the space characteristic function $\chi_{B_R}$ in the same way but with 
$H_{\eta}(\br) = \frac{1}{\eta^3} H(\frac{\br}{\eta})$ 
for a standard smooth kernel $H\in C_c^{\infty}(\mathbb{R}^3)$ such that $\int_{\mathbb{R}^3}H(\br)\, dV = 1$ and is supported in the unit ball $\{|\br|\le 1\}$.
Then, we define the restriction 
\begin{equation}
    \varphi^{\tau}_{\epsilon, R} \coloneqq \varphi|_{\bar{\Omega}\times[0,T)}
\end{equation}
and by definition $\varphi^{\tau}_{\epsilon, R} \in \bar{D}(\bar{\Omega}\times[0,T))$.
It follows that $0\le\varphi^{\tau}_{\epsilon, R}\le 1$ everywhere and $\varphi^{\tau}_{\epsilon, R}$ does not necessarily vanish at $t=0$ or on $\partial B$.  
Since $Q\ge 0$, the equality \eqref{fgRotationalLimit} yields the following inequality 
\begin{equation}
    \begin{aligned} 
        &
        -\int_0^T\int_{\Omega}\frac{1}{2}\partial_t\varphi^{\tau}_{\epsilon, R}|\bu_{\omega}|^2\, dV dt
        -\int_0^T\int_{\Omega}\grad\varphi^{\tau}_{\epsilon, R}\bdot
        \left[\frac{1}{2}|\bu_{\omega}|^2\bu+p_{\omega}\bu_{\omega}\right]\, dV dt \\ 
        & \hspace{30pt} \le \int_{\Omega}\frac{1}{2}|\bu_{\omega}(0)|^2\, dV dt
        -\int_0^T\int_{\Omega}\varphi^{\tau}_{\epsilon, R}\grad\mathbf{u}_{\phi}\bdots
        \mathbf{u}_{\omega}\otimes\mathbf{u}_{\omega}\, dV dt
    \end{aligned}\label{fgRotationalLimit2}
\end{equation}
\red{Here we have used the fact that $\langle\bu_\phi\bdot \btau_w,\psi\rangle
=\int_{-\delta}^T \langle\bu_\phi\bdot \btau_w(\cdot,t),1\rangle\, 
\chi_{(-\delta,\tau]}^{\epsilon}(t) dt=0$ under the hypothesis of our theorem.}

For sufficiently small $\epsilon$ s.t. $\tau + \epsilon < T$, 
$\chi_{(-\delta,\tau]}^{\epsilon}(t) = 1$ for $t\in[-\delta + \epsilon,\tau-\epsilon)$ and 0 for $t>\tau + \epsilon$.
Thus, we have that
\begin{equation*}
    \partial_t\chi_{(-\delta,\tau]}^{\epsilon} = \chi_{(-\delta,\tau]}*\partial_tG_{\epsilon}
    = -\frac{1}{\epsilon}G(\frac{t-\tau}{\epsilon}) = -G_\epsilon(t-\tau)
\end{equation*}
and 
\begin{equation*}
    \text{supp}(\partial_t\chi_{(-\delta,\tau]}^{\epsilon})\cap(0,T) = (\tau-\epsilon, \tau+\epsilon)
\end{equation*}
Hence, 
\begin{equation}
    \begin{aligned}
        \int_0^T\int_{\Omega}\frac{1}{2}\partial_t\varphi^{\tau}_{\epsilon, R}|\bu_{\omega}|^2 \, dV dt
        & = -\int_0^TG_{\epsilon}(t-\tau)\frac{1}{2}\int_{\Omega}\chi_{B_R}^{\eta}(\bx)|\bu_{\omega}(\bx,t)|^2 \, dVdt\\
        & = -\check{G}_{\epsilon}*I(\tau)
    \end{aligned}\label{local_t_energy}
\end{equation}
where $\check{G}_\epsilon(t) \coloneqq G_\epsilon(-t)$, and $I(t) \coloneqq \frac{1}{2}\int_{\Omega}\chi_{B_R}^{\eta}(\bx)|\bu_{\omega}|^2(\bx, t)\, dV$ which is a integrable function in $t$ by Fubini theorem.
Furthermore, by a general result of approximation to the identity (i.e. Theorem 3.2.1. of \cite{stein2009real}), we have for a.e. $\tau\in(0,T)$,
\begin{equation}
    \check{G}_{\epsilon}*I(\tau)\xrightarrow{\epsilon\to0}I(\tau)\label{ae_conv}
\end{equation}
Then it follows that  
\begin{equation}
    \lim_{\epsilon\to0}\int_0^T\int_{\Omega}\frac{1}{2}\partial_t\varphi^{\tau}_{\epsilon, R}|\bu_{\omega}|^2 \, dV dt
    = -\frac{1}{2}\int_{\Omega}\chi_{B_R}^{\eta}(\bx) |\bu_{\omega}(\bx, \tau)|^2 \, dV
\end{equation}
which further converges as $R\to\infty$ by monotonicity 
\begin{equation}
    \lim_{R\to\infty}\lim_{\epsilon\to0}\int_0^T\int_{\Omega}
    \frac{1}{2}\partial_t\varphi^{\tau}_{\epsilon, R}|\bu_{\omega}|^2 \, dV dt
    = -\frac{1}{2}\int_{\Omega}|\bu_{\omega}(\bx, \tau)|^2 \, dV
\end{equation}  

Now we look at the flux term involving the spatial gradient. 
\begin{equation}
    \grad\varphi^{\tau}_{\epsilon, R} = 
    \chi_{(-\delta, \tau]}^{\epsilon}\grad\chi_{B_R}^{\eta} 
    = \chi_{(-\delta, \tau]}^{\epsilon}\chi_{B_R} * \grad H_\eta
\end{equation}
which is only supported on $A_{R,\eta} \coloneqq B_{R+\eta}\backslash B_{R-\eta}$, 
an annulus of thickness $2\eta$, 
and $\grad\varphi^{\tau}_{\epsilon, R}$ is bounded by $\frac{C}{\eta}$ uniformly in $R$.
Then, by H\"older inequality 
\begin{align*}
    &\int_0^T\int_{\Omega}\grad\varphi^{\tau}_{\epsilon, R}\bdot
    \left[\frac{1}{2}|\bu_{\omega}|^2\bu+p_{\omega}\bu_{\omega}\right]\, dV dt\\
    & \hspace{30pt} \le C\bigg(\norm{\bu_{\omega}}_{L^3(0,T;L^3(A_{R,\eta}))}^3
    + \norm{\bu_{\omega}}_{L^2(0,T;L^2(A_{R,\eta}))}^2
    \cdot\norm{\bu_{\phi}}_{L^{\infty}(0,T;L^{\infty}(A_{R,\eta}))}\\
    &\hspace{60pt} + \norm{p_{\omega}}_{L^{q}(0,T;L^{q}(A_{R,\eta}))}\cdot
    \norm{\bu_{\omega}}_{L^{q'}(0,T;L^{q'}(A_{R,\eta}))}\bigg)\\
    &\hspace{60pt} \xrightarrow{R\to\infty} 0
\end{align*}
for some $q\in(\frac{3}{2},2)$ and $\frac{1}{q}+\frac{1}{q'}=1.$
The upper bound above goes to 0 since $p_{\omega}$ is globally bounded in spacetime $L^{q}$ and $\bu_\omega$ is globally bounded in spacetime $L^r$ for any $r\in[2,3],$ by interpolation between $L^2$ and $L^3$ assumed in \eqref{assumption:ext-vel-global}. Here $q' = \frac{q}{q-1}\in (2,3)$ for $q\in(3/2, 2)$.

Finally, the global boundedness of $\bu_{\omega}$ gives
\begin{equation}
    \int_0^T\int_{\Omega}|\varphi^{\tau}_{\epsilon, R}\grad\mathbf{u}_{\phi}\bdots
    \mathbf{u}_{\omega}\otimes\mathbf{u}_{\omega}|\, dV dt 
    \le \norm{\grad\bu_{\phi}}_{L^{\infty}(\Omega\times(0,T))}
    \cdot\norm{\bu_{\omega}}_{L^2(0,T;L^2(\Omega))}^2
\end{equation}
Given $|\varphi^{\tau}_{\epsilon, R}f| \le |f|$ with $f = \grad\bu_\phi:\bu_\omega\otimes\bu_\omega$ integrable on $\Omega\times(0,T)$ because of assumption \eqref{assumption:ext-vel-global}, 
it follows from dominated convergence theorem that 
\begin{equation}
    \lim_{R\to\infty}\lim_{\epsilon\to0}
    \int_0^T\int_{\Omega}\varphi^{\tau}_{\epsilon, R}\grad\mathbf{u}_{\phi}\bdots
    \mathbf{u}_{\omega}\otimes\mathbf{u}_{\omega}\, dV dt 
    = \int_0^{\tau}\int_{\Omega}\grad\mathbf{u}_{\phi}\bdots
    \mathbf{u}_{\omega}\otimes\mathbf{u}_{\omega}\, dV dt
\end{equation}
Therefore, as $\epsilon\to0$ and $R\to\infty$, we obtain from the local inequality \eqref{fgRotationalLimit2} the global result that, for a.e. $\tau\in[0,T]$,
\begin{align*} 
    E_{\omega}(\tau)
    &\le E_{\omega}(0) 
    -\int_0^{\tau}\int_{\Omega}\grad\mathbf{u}_{\phi}\bdots
    \mathbf{u}_{\omega}\otimes\mathbf{u}_{\omega}\, dV dt\\
    &\le E_{\omega}(0) + 
    C\int_0^{\tau}\norm{\grad\bu_{\phi}(s)}_{L^{\infty}(\Omega)}E_{\omega}(t)dt
\end{align*}
where the second inequality is deduced from 
Cauchy-Schwartz. 
Thus by Gr\"onwall's inequality,
\begin{equation}
    E_{\omega}(\tau) \le E_{\omega}(0)\exp\left(C\int_0^{\tau}\norm{\grad\bu_{\phi}(s)}_{L^{\infty}(\Omega)}dt\right), \quad {\rm a.e. } \tau\in (0,T). 
\end{equation}
Since $\bu_{\omega}(0) = \bu(0) - \bu_{\phi} = \bu_0 - \bu_0 = \bzed$, $E_{\omega}(0) = 0$ and
$E_{\omega}(\tau) = 0$ for a.e. $\tau\in(0,T)$. Therefore, we can conclude that $\bu_{\omega} = \bzed$ and thus $\bu = \bu_{\phi}$
almost everywhere in $\Omega\times(0,T)$.

\section{Proof of Theorem \ref{thm:internal}}\label{sec:proof-thm-3} 
The proof is based on the concept of a dissipative solution of Euler up to the boundary in the sense 
of Lions-Bardos-Titi \cite{lions1996mathematical,bardos2013mathematics}, which is defined to be a $\bu\in L^2([0,T],H(\Omega))$
such that for every divergence-free test vector field $\bw\in C^1(\Bar{\Omega}\times[0,T])$ with $\bw|_{\partial\Omega}\cdot\bn = 0$, the following inequality holds
\beal
    \int_{\Omega}|\bu(\bx, t)&-\bw(\bx, t)|^2\, dV
    \le e^{\int_0^t2\norm{S(\bw)}_{L^{\infty}(\Omega)}ds}\int_{\Omega}|\bu(\bx, 0)-\bw(\bx, 0)|^2\, dV\\
    &+2\int_0^t\int_{\Omega} e^{\int_s^t2\norm{S(\bw)}_{L^{\infty}(\Omega)}d\tau}
    E(\bw(\bx, s))\cdot(\bu(\bx, s)-\bw(\bx, s))\, dV ds. \lb{dissEbdry}
\eeal
Here $S(\bw) = (\grad\bw+\grad\bw^\top)/2$ and the Euler residual is defined by 
\begin{align}
E(\bw) = -\partial_t\bw - \mathbb{P}((\bw\cdot\grad)\bw)\label{def:s-and-w}
\end{align}
with $\mathbb{P}$ denoting the Leray-Helmholtz projection on $H(\Omega)$.
Weak-strong uniqueness in this class of dissipative solutions is immediate:  
see \cite{bardos2013mathematics}, Definition 4.1 and Remark 3.1. 

A useful fact is the following
\begin{lemma}\label{lemma1}
An admissible weak Euler solution $\bu\in L^2([0,T],H(\Omega))$ satisfying 
\begin{align}
    \frac{d}{dt}\int_{\Omega}\bu\cdot\bw \, dV = \int_{\Omega}(S(\bw)(\bu-\bw)\cdot(\bu-\bw)-E(\bw)\cdot\bu)\, dV\label{key}
\end{align}
in the sense of distributions for every divergence-free field $\bw\in C^1(\Bar{\Omega}\times[0,T])$ with $\bw|_{\partial\Omega}\cdot\bn = 0$ is a dissipative solution of Euler up to the boundary \eqref{dissEbdry}. 
\end{lemma}
\begin{proof}
See Section 7 of \cite{bardos2014non}. 
\end{proof}
\noindent In \cite{de2010admissibility}, 
the identity \eqref{key} is proved for the case that 
$\bw$ is compactly supported in $\Omega$ at almost every time. 
Now we consider a divergence-free test vector field $\bw\in C^1(\Bar{\Omega}\times[0,T])$ 
with boundary condition $\bw|_{\partial\Omega}\bdot\bn = 0$, 
which does not necessarily have compact support in $\Omega$. 
We follow the approach of \cite{bardos2014non, wiedemann2018weak} to approximate $\bw$ with vector fields that do have compact support. 

For this purpose, we need the following result on existence of solutions to the div-curl problem:
\begin{lemma}\label{lemma:div-curl}
    Let $\Omega$ be a bounded, simply-connected domain in $\mathbb{R}^n$ with $n=2,3$ and with $C^{\infty}$ boundary. Consider a divergence-free vector field $\bw\in C^1(\bar\Omega\times[0,T])$ with boundary condition $\bw\cdot\bn = 0$. Then for $n=3$, there exist a vector stream function ${\boldsymbol\Psi}\in C(0,T;C^{1,\alpha}(\bar{\Omega})\cap C^1(\bar\Omega\times[0,T])$ for $0<\alpha<1$ such that
    \beal
        \grad\btimes{\boldsymbol \Psi} &= \bw \quad\text{ in }\Omega\\
        \grad\bdot{\boldsymbol \Psi} &= 0 \quad\text{ in }\Omega\\
        \bn\btimes\boldsymbol\Psi &= \mathbf{0}\quad\text{ on } \partial\Omega\label{eq:div-curl-3d}
    \eeal
For $n=2$, ${\boldsymbol\Psi}=\psi\hat{{\bf z}}$ for a scalar stream function $\psi$ satisfying 
$\psi=0$ on $\partial\Omega.$
\end{lemma}
\begin{proof}
    This follows from results of \cite{von2006necessary} and especially 
    \cite{kress1969grundzuege}. Note that our assumption of simply-connectedness 
    means that the domains $\Omega$ have no handles or, equivalently, first Betti number equal to zero.
    Theorems 5.1 and 5.2 of \cite{kress1969grundzuege} state that, for any divergence-free 
    $\bw\in C^\alpha(\bar{\Omega})$ for some $0<\alpha<1$ satisfying $\bw\bdot\bn=0,$ there exists 
    $\boldsymbol\Psi\in C^{1,\alpha}(\bar\Omega)$ which 
    solves \eqref{eq:div-curl-3d} and which is unique subject to the additional constraint that 
    $\int_{\partial\Omega} \bPsi\bdot\bn \, dA=0.$
     Considering $\bw\in C^1(\bar\Omega\times[0,T]),$ we may 
    apply this result for every time $t\in[0,T]$ and conclude by stability that $\boldsymbol\Psi\in C(0,T;C^{1,\alpha}(\bar\Omega))\cap C^1(\bar\Omega\times[0,T])$.
\end{proof}

Now let $\chi:[0,\infty)\to\mathbb{R}$ be $C^{\infty}$-smooth cutoff function s.t. 
$0\le\chi\le1$ and 
\begin{equation}
    \chi(s) = 
        \begin{cases}
            0 &\text{ if } s<1\\
            1 &\text{ if } s>2
        \end{cases}\label{def:smooth_cutoff}
\end{equation}
and let $\epsilon > 0$ and
\begin{equation}
    \bw_{\epsilon}(\bx, t) = 
    \grad\times\left(\chi\left(\frac{d(\bx) }{\epsilon}\right)\Psi(\bx, t)\right)
\end{equation}
where recall that the distance function $d$ is $C^\infty$ in a tubular neighborhood 
$\Omega_{\eta}$ for some $\eta>0.$ Hence, $\bw_{\epsilon}\in C^1(0,T; C_c^1(\Omega))$ and $\partial_t\bw_\epsilon\in  C(0,T; C_c^1(\Omega))$ for sufficiently small $\epsilon>0$ and $\bw_\epsilon$ satisfies \eqref{key}:
\begin{align}
    \frac{d}{dt}\int_{\Omega}\bu\cdot\bw_{\epsilon} \, dV  
    &= \int_{\Omega}(S(\bw_{\epsilon})(\bu-\bw_{\epsilon})\cdot(\bu-\bw_{\epsilon})
    -E(\bw_{\epsilon})\cdot\bu)\, dV \label{key_eps}\\
    &= \int_{\Omega}[\partial_t\bw_{\epsilon}\cdot\bu 
        + (\bu\cdot\grad\bw_{\epsilon})\cdot\bu 
        - ((\bu-\bw_{\epsilon})\cdot\grad\bw_{\epsilon})\cdot\bw_{\epsilon}]\, dV \\ 
    &= \int_{\Omega}[\partial_t\bw_{\epsilon}\cdot\bu 
        + (\bu\cdot\grad\bw_{\epsilon})\cdot\bu ]\, dV           
        \label{key_eps_equiv}
\end{align}
In the last step, we used $\bu-\bw_{\epsilon}\in H(\Omega)$, so that it 
follows from \eqref{eq:characterization-H-Omega-1} that
\begin{equation}
    \int_{\Omega}((\bu-\bw_{\epsilon})\cdot\grad\bw_{\epsilon})\cdot\bw_{\epsilon} \, dV  
    = \int_{\Omega}(\bu-\bw_{\epsilon})\cdot\frac{1}{2}\grad|\bw_{\epsilon}|^2\, dV  = 0\label{eq:conv-end}
\end{equation}
The same simplifications apply also to \eqref{key} with $\bw_\epsilon$ replaced by $\bw$
and this alternative form of \eqref{key} analogous to \eqref{key_eps_equiv} is most convenient 
to apply Lemma \ref{lemma1}.

To this end, we want to send $\epsilon\to0$ in \eqref{key_eps_equiv} to recover \eqref{key}. 
By definition, we can rewrite $\bw_\epsilon$ as follows:
\begin{equation}
    \bw_{\epsilon} = \chi\left(\frac{d(\bx) }{\epsilon}\right)\grad\btimes \bPsi 
    + \frac{1}{\epsilon}\chi'\left(\frac{d(\bx) }{\epsilon}\right)\grad d\btimes \bPsi\label{w_eps}
\end{equation}
Furthermore, given that $\bPsi\in C^1(\bar\Omega\times[0,T])$ and $\bn\btimes\bPsi|_{\partial\Omega} = \bzed$, 
there exists a constant $C$ such that
\begin{equation}
    |\bn(\bx)\btimes\bPsi(\bx, t)| \le Cd(\bx) \le C\epsilon
\end{equation}
for all $\bx\in\overline{\Omega}_{2\epsilon}$. 
Together with the observations that the support of $\chi'\left(\frac{d(\bx) }{\epsilon}\right)$ 
is contained in $(\epsilon,2\epsilon)$ and $|\grad d| = |\bn| = 1$, we obtain convergence in $L^{\infty}(0,T;L^2(\Omega))$ of the second term in \eqref{w_eps} to zero, and thus we get
\begin{eqnarray}
    \bw_{\epsilon} \to \bw \text{ strongly in } L^{\infty}([0,T];L^2(\Omega)), \;\; 
    \text{as } \epsilon\to0\label{w-L2Conv}\\
    \partial_t\bw_{\epsilon} \to \partial_t\bw \text{ strongly in } L^{\infty}([0,T];L^2(\Omega)), \;\; 
    \text{as } \epsilon\to0\label{time-derivative-L2Conv}
\end{eqnarray}
Subsequently for the terms involving time derivative in \eqref{key_eps} and \eqref{key_eps_equiv}, it follows from \eqref{w-L2Conv} and \eqref{time-derivative-L2Conv} respectively that 
\begin{eqnarray}
    \frac{d}{dt}\int_{\Omega}\bu\cdot\bw_{\epsilon}\, dV \to\frac{d}{dt}\int_{\Omega}\bu\cdot\bw \, dV \label{eq:conv-start}\\
    \int_{\Omega}\partial_t\bw_{\epsilon}\cdot\bu \, dV \to \int_{\Omega}\partial_t\bw\cdot\bu \, dV 
\end{eqnarray}
in the sense of distribution in time. Now it remains to show that
\begin{equation}
    \int_{\Omega}(\bu\cdot\grad\bw_{\epsilon})\cdot\bu \, dV  \to \int_{\Omega}(\bu\cdot\grad\bw)\cdot\bu \, dV  
\end{equation}
in the sense of distribution in time, as $\epsilon\to0$. Here we need to perform a local analysis in the region near the boundary. For $n=3$, $\partial\Omega$ is a 2-dimensional smooth manifold without boundary.
As a consequence of Poincaré-Hopf theorem, however,
there does not exist any tangent vector on $\partial\Omega$ that is non-vanishing everywhere, thus a global parametrization of the boundary is impossible.
To resolve this issue, we look at a subset of $\Omega_{2\epsilon}$ so that 
there exists a well-defined local coordinate in terms of tangent vectors and normal vectors.
Consider some point $\bx^0\in\partial\Omega$ and some $0< r < 2\epsilon$. 
Let $\Omega_{2\epsilon}^0 = B(\bx^0, r)\cap\Omega$. 
Then for every $\bx\in\Omega_{2\epsilon}^0$, let $\hat{\bx} = \pi(\bx) $, 
where $\pi:\Omega_{2\epsilon}^0\to\partial\Omega$ is the smooth projection map 
for sufficiently small $\epsilon$. 
Since $\partial\Omega$ is $C^\infty$, there exist tangent vectors $\btau^0, \btau^1$ on $\partial\Omega$ such that 
$(\btau^0,\btau^1,\bn)$ is a $C^\infty$ smooth orthogonal frame on 
$\partial\Omega\cap\partial\Omega_{2\epsilon}^0$. Here indices $0$ and $1$ belong to $(\mathbb{Z}_2, +)$ such that $0+1 = 1+0 = 1$ and $1+1 = 0$. With this index notation,
$\btau^0\times\btau^1 = \bn$ and $\btau^i\times\bn = (-1)^{i+1}\btau^{i+1}$.
For $\bx\in\Omega_{2\epsilon}^0$, we denote $u_{\tau}^i(\bx)  = \bu(\bx) \bdot\btau^i(\hat{\bx})$ for $i=0,1$, $w_n(\bx) = \bw(\bx)\bdot\bn(\hat{\bx}),$ $\partial_{\tau}^i w_n(\bx)  = \grad w_{n}(\bx) \cdot\btau^{i}(\hat{\bx})$ and so on. 
Then we write
\beal
    &\int_{\Omega_{2\epsilon}^0}\bu\cdot\grad(\bw_{\epsilon}-\bw)\cdot\bu \, dV \\ 
    =& \int_{\Omega_{2\epsilon}^0}\partial_n(\bw_{\epsilon}-\bw)_n u_n u_{n}\, dV  
    +\int_{\Omega_{2\epsilon}^0}\partial_n(\bw_{\epsilon}-\bw)_{\tau}^0 u_{n} u_{\tau}^0\, dV 
    +\int_{\Omega_{2\epsilon}^0}\partial_n(\bw_{\epsilon}-\bw)_{\tau}^1 u_{n} u_{\tau}^1\, dV \\
    +&\int_{\Omega_{2\epsilon}^0}\partial_{\tau}^0(\bw_{\epsilon}-\bw)_n u_{\tau}^0 u_n\, dV  
    +\int_{\Omega_{2\epsilon}^0}\partial_{\tau}^0(\bw_{\epsilon}-\bw)_{\tau}^0 u_{\tau}^0 u_{\tau}^0\, dV 
    +\int_{\Omega_{2\epsilon}^0}\partial_{\tau}^0(\bw_{\epsilon}-\bw)_{\tau}^1 u_{\tau}^0 u_{\tau}^1\, dV \\
    +&\int_{\Omega_{2\epsilon}^0}\partial_{\tau}^1(\bw_{\epsilon}-\bw)_n u_{\tau}^1 u_n\, dV  
    +\int_{\Omega_{2\epsilon}^0}\partial_{\tau}^1(\bw_{\epsilon}-\bw)_{\tau}^0 u_{\tau}^1 u_{\tau}^0\, dV 
    +\int_{\Omega_{2\epsilon}^0}\partial_{\tau}^1(\bw_{\epsilon}-\bw)_{\tau}^1 u_{\tau}^1 u_{\tau}^1\, dV \\
    \eqqcolon & \quad
    I_{n,n} + I_{n,0} + I_{n,1} + 
    I_{0,n} + I_{0,0} + I_{0,1} + 
    I_{1,n} + I_{1,0} + I_{1,1}\label{eq:boundary-terms}
\eeal

We next compute all components and derivatives of $\bw_\epsilon - \bw$ using \eqref{w_eps}
\begin{align}
    (\bw_{\epsilon}-\bw)_{n} &= 
    \left(\chi\left(\frac{d}{\epsilon}\right)-1\right)w_{n}\label{derivatives1} \\
    (\bw_{\epsilon}-\bw)_{\tau}^i &= 
    \left(\chi\left(\frac{d}{\epsilon}\right)-1\right) w_{\tau}^i
    + \frac{1}{\epsilon}\chi'\left(\frac{d}{\epsilon}\right)(-1)^{i+1}\Psi_{\tau}^{i+1}\\
    \partial_{n}(\bw_{\epsilon}-\bw)_{n} &= 
     \frac{1}{\epsilon}\chi'\left(\frac{d}{\epsilon}\right) w_n 
    +\left(\chi\left(\frac{d}{\epsilon}\right)-1\right)\partial_n w_{n} \\
    \partial_{n}(\bw_{\epsilon}-\bw)_{\tau}^i &= 
    \frac{1}{\epsilon}\chi'\left(\frac{d}{\epsilon}\right)w_\tau^i + \left(\chi\left(\frac{d}{\epsilon}\right)-1\right)\partial_{n} w_{\tau}^i\\
    \nonumber
    &+ \frac{1}{\epsilon^2}\chi''\left(\frac{d}{\epsilon}\right)(-1)^{i+1}\Psi_{\tau}^{i+1}
    + \frac{1}{\epsilon}\chi'\left(\frac{d}{\epsilon}\right)(-1)^{i+1}\partial_n \Psi_{\tau}^{i+1}\\
    \partial_{\tau}^i(\bw_{\epsilon}-\bw)_{n} &= 
    \left(\chi\left(\frac{d}{\epsilon}\right)-1\right)\partial_{\tau}^i w_{n}\\
    \partial_{\tau}^j(\bw_{\epsilon}-\bw)_{\tau}^i &= 
    \left(\chi\left(\frac{d}{\epsilon}\right)-1\right)\partial_{\tau}^j w_{\tau}^i
    + \frac{1}{\epsilon}\chi'\left(\frac{d}{\epsilon}\right)
    (-1)^{i+1}\partial_{\tau}^{j}\Psi_{\tau}^{i+1}\label{derivatives2}
\end{align}
We list some observations useful for estimation of the various terms in \eqref{eq:boundary-terms}.
Recalling the assumption \eqref{cond:no-flow-through},  
wall-normal velocity $u_n$ vanishes uniformly as it approaches the boundary.
Moreover, since $\bw\in C^{1}(\Bar{\Omega}\times[0,T])$ 
and $w_{n} = 0$ on $\partial\Omega$, 
there is likewise a constant independent of $t$ such that 
\begin{align}
    |w_n(\bx) |\le Cd(\bx)  \le C\epsilon \;\;\;\; &\text{in } \Omega_{2\epsilon}^0
\end{align}
Similarly, the fact that $\bPsi\in C([0,T];C^{1,\alpha}(\Bar{\Omega}))$ 
and $\bn\btimes\bPsi|_{\partial\Omega}\equiv \bzed$ implies that
\begin{align}
    &\Psi_k \equiv 0,\ \partial_{\tau}^j \Psi_k \equiv 0, \;\;\;\; \text{on } \partial\Omega\\
    &|\Psi_k(\bx) |\le Cd(\bx) \le C\epsilon,\ |\partial_{\tau}^j \Psi_k(\bx) | \le Cd(\bx)^\alpha \le C\epsilon^\alpha \;\;\;\; 
    \text{in } \Omega_{2\epsilon}^0
\end{align}
for all $k,j\in\{0,1\}$, 
where $\Psi_0 = \Psi_{\tau}^0$ and $\Psi_1 = \Psi_{\tau}^1$ and $\Psi_2 = \Psi_n$.
Finally, note that $\bPsi,\bw,\bu$ are uniformly bounded on $\Omega_{2\epsilon}^0$ for small $\epsilon<\delta$, 
and that there is a constant independent of $\epsilon$ such that $|\Omega_{2\epsilon}^0|\le C\epsilon$.

With these observations above, 
we use \eqref{derivatives1}-\eqref{derivatives2} to get the following uniform-in-time estimates
\beal
    |I_{n,n}|&\le\frac{1}{\epsilon}\int_{\Omega_{2\epsilon}^0}|u_n|^2\norm{\chi'}_{\infty}\norm{w_n}_{L^\infty(\Omega_{2\epsilon})}\, dV \\
    &\qquad\qquad + \int_{\Omega_{2\epsilon}^0}|u_n|^2\norm{\chi-1}_{\infty}\norm{\partial_n w_n}_{\infty}\, dV \\
    &\le C\norm{u_n}_{L^{\infty}(\Omega_{2\epsilon})}^2(\epsilon+\epsilon)\to0
\eeal
\beal
    |I_{n,i}|&\le
    \frac{1}{\epsilon}\int_{\Omega_{2\epsilon}^0}|u_n||u_\tau^i|\norm{\chi'}_{\infty}|w_\tau^i|\;\, dV + 
    \int_{\Omega_{2\epsilon}^0}|u_n||u_{\tau}^i|\norm{\chi-1}_{\infty}\norm{\partial_nw_{\tau}}_{\infty}\, dV  \\
    &\qquad\qquad+ \frac{1}{\epsilon^2}\int_{\Omega_{2\epsilon}^0}|u_n||u_{\tau}^i|\norm{\chi''}_{\infty}|\Psi_\tau|\, dV 
    +\frac{1}{\epsilon}\int_{\Omega_{2\epsilon}^0}|u_n||u_{\tau}^i|\|\chi'\|_{\infty}\left|\frac{\partial \Psi_\tau}{\partial n}\right|\, dV\\
    &\le C\norm{u_n}_{L^{\infty}(\Omega_{2\epsilon})}(1 + \epsilon+1 + 1)\to0
\eeal

\begin{align}
    |I_{i,n}|\le\int_{\Omega_{2\epsilon}^0}|u_n||u_{\tau}^i|\norm{\chi-1}_{\infty}\norm{\partial_{\tau}^i w_{n}}_{\infty}\, dV \le C\epsilon\norm{u_n}_{L^{\infty}(\Omega_{2\epsilon})}\to0
\end{align}

\beal
    |I_{i,j}|&\le \int_{\Omega_{2\epsilon}^0}|u_{\tau}^i||u_{\tau}^j||\chi-1|_{\infty}\norm{\partial_{\tau}w_{\tau}}_{\infty}\, dV  \\
    &\qquad\qquad + \frac{1}{\epsilon}\int_{\Omega_{2\epsilon}^0}|u_{\tau}^i||u_{\tau}^j|\norm{\chi'}_{\infty}\max_{k\in\{0,1\}}|\partial_{\tau}^j \Psi_k|\, dV \\
    &\le C\epsilon + C\epsilon^\alpha\to0
\eeal 
This shows that as $\epsilon\to0$,
\begin{align}
    \int_{\Omega_{2\epsilon}^0}|\bu\cdot\grad(\bw_{\epsilon}-\bw)\cdot\bu|\;\, dV  \to0
\end{align}

Since $\partial\Omega$ is compact, there exist finitely many points $\bx^i\in\partial\Omega$ for $i=1,\dots,N$ 
such that 
\begin{align}
    \Omega_{2\epsilon} = \cup_{i=1}^N\Omega_{2\epsilon}^i
\end{align}
By the same argument as that for $\Omega_{2\epsilon}^0$, we can show that
\begin{align}
    \int_{\Omega_{2\epsilon}^i}|\bu\cdot\grad(\bw_{\epsilon}-\bw)\cdot\bu|\;\, dV  \to0
\end{align}
for all $i=1,\dots,N$. Therefore, we complete the proof with
\begin{align}
    \int_{\Omega}\bu\cdot\grad(\bw_{\epsilon}-\bw)\cdot\bu\;\, dV  \to0
\end{align}
The proof for $n=2$ is similar but even simpler. 

\section*{Acknowledgements} 
We thank T. Drivas for very helpful discussions. 
G.E. thanks also the Department of Physics of the University of Rome `Tor Vergata' 
for hospitality while this work was finalized. 

\section*{Declarations}

\subsection*{Funding}
This work was funded by the Simons Foundation, Targeted Grant MPS-663054 and Collaboration Grant MPS-1151713. G.E. also acknowledges support from the European Research Council (ERC) under the European Union’s Horizon 2020 research and innovation program (Grant Agreement No. 882340).

\subsection*{Conflicts of Interest/Competing Interests}
The authors declare that they have no conflict of interest or competing interests.

\subsection*{Data Availability}
This work does not have associated data.

\bibliographystyle{plain}
\bibliography{bibliography}

\begin{thebibliography}{10}

\bibitem{achenbach1972experiments}
Elmar Achenbach.
\newblock Experiments on the flow past spheres at very high {R}eynolds numbers.
\newblock {\em Journal of fluid mechanics}, 54(3):565--575, 1972.

\bibitem{bardos2014non}
Claude Bardos, L~Sz{\'e}kelyhidi, and Emil Wiedemann.
\newblock Non-uniqueness for the {E}uler equations: the effect of the boundary.
\newblock {\em Russian Mathematical Surveys}, 69(2):189, 2014.

\bibitem{bardos2018onsager}
Claude Bardos and Edriss~S Titi.
\newblock Onsager’s conjecture for the incompressible {E}uler equations in
  bounded domains.
\newblock {\em Archive for Rational Mechanics and Analysis}, 228(1):197--207,
  2018.

\bibitem{bardos2019onsager}
Claude Bardos, Edriss~S Titi, and Emil Wiedemann.
\newblock Onsager’s conjecture with physical boundaries and an application to
  the vanishing viscosity limit.
\newblock {\em Communications in Mathematical Physics}, 370:291--310, 2019.

\bibitem{bardos2013mathematics}
Claude~W Bardos and Edriss~S Titi.
\newblock Mathematics and turbulence: where do we stand?
\newblock {\em Journal of Turbulence}, 14(3):42--76, 2013.

\bibitem{batchelor1967fluid}
G.~K. Batchelor.
\newblock {\em An Introduction to Fluid Dynamics}.
\newblock Cambridge Mathematical Library. Cambridge University Press, 1967.

\bibitem{borchers1990equations}
Wolfgang Borchers and Hermann Sohr.
\newblock On the equations rot v= g and div u= f with zero boundary conditions.
\newblock {\em Hokkaido Mathematical Journal}, 19(1):67--87, 1990.

\bibitem{brenier2011weak}
Yann Brenier, Camillo De~Lellis, and L{\'a}szl{\'o} Sz{\'e}kelyhidi.
\newblock Weak-strong uniqueness for measure-valued solutions.
\newblock {\em Communications in mathematical physics}, 305:351--361, 2011.

\bibitem{chatzimanolakis2022vortex}
Michail Chatzimanolakis, Pascal Weber, and Petros Koumoutsakos.
\newblock Vortex separation cascades in simulations of the planar flow past an
  impulsively started cylinder up to {R}e$=100\, 000$.
\newblock {\em Journal of Fluid Mechanics}, 953:R2, 2022.

\bibitem{dalembert1749theoria}
Jean le~Rond d{'}Alembert.
\newblock Theoria resistentiae quam patitur corpus in fluido motum, ex
  principiis omnino novis et simplissimis deducta, habita ratione tum
  velocitatis, figurae, et massae corporis moti, tum densitatis compressionis
  partium fluidi.
\newblock manuscript at Berlin-Brandenburgische Akademie der Wissenschaften,
  Akademie-Archiv call number: I-M478, 1749.

\bibitem{dalembert1768paradoxe}
Jean le~Rond d{'}Alembert.
\newblock Paradoxe propos\'e aux g\'eom\`etres sur la r\'esistance des fluides.
\newblock in: Opuscules math\'ematiques, vol. 5 (Paris), Memoir XXXIV, Section
  I, 132--138, 1768.

\bibitem{de2010admissibility}
Camillo De~Lellis and L{\'a}szl{\'o} Sz{\'e}kelyhidi.
\newblock On admissibility criteria for weak solutions of the {E}uler
  equations.
\newblock {\em Archive for rational mechanics and analysis}, 195(1):225--260,
  2010.

\bibitem{diperna1987oscillations}
Ronald~J DiPerna and Andrew~J Majda.
\newblock Oscillations and concentrations in weak solutions of the
  incompressible fluid equations.
\newblock {\em Communications in mathematical physics}, 108(4):667--689, 1987.

\bibitem{drivas2018nguyen}
Theodore~D Drivas and Huy~Q Nguyen.
\newblock Onsager's conjecture and anomalous dissipation on domains with
  boundary.
\newblock {\em SIAM Journal on Mathematical Analysis}, 50(5):4785--4811, 2018.

\bibitem{drivas2019remarks}
Theodore~D Drivas and Huy~Q Nguyen.
\newblock Remarks on the emergence of weak {E}uler solutions in the vanishing
  viscosity limit.
\newblock {\em Journal of Nonlinear Science}, 29(2):709--721, 2019.

\bibitem{eyink2024onsager}
Gregory Eyink.
\newblock Onsager's ‘ideal turbulence’theory.
\newblock {\em Journal of Fluid Mechanics}, 988:P1, 2024.

\bibitem{eyink2008dissipative}
Gregory~L Eyink.
\newblock Dissipative anomalies in singular {E}uler flows.
\newblock {\em Physica D: Nonlinear Phenomena}, 237(14-17):1956--1968, 2008.

\bibitem{eyink2021josephson}
Gregory~L. Eyink.
\newblock Josephson-{A}nderson relation and the classical d'{A}lembert paradox.
\newblock {\em Phys. Rev. X}, 11:031054, Sep 2021.

\bibitem{eyink2021jaerratum}
Gregory~L. Eyink.
\newblock Erratum: {J}osephson-{A}nderson relation and the classical
  d'{A}lembert paradox [{P}hys. {R}ev. {X} 11, 031054 (2021)].
\newblock {\em Phys. Rev. X}, 14:039901, Aug 2024.

\bibitem{fehn2022numerical}
Niklas Fehn, Martin Kronbichler, Peter Munch, and Wolfgang~A Wall.
\newblock Numerical evidence of anomalous energy dissipation in incompressible
  {E}uler flows: towards grid-converged results for the inviscid
  {T}aylor--{G}reen problem.
\newblock {\em Journal of Fluid Mechanics}, 932:A40, 2022.

\bibitem{frisch1995turbulence}
Uriel Frisch.
\newblock {\em Turbulence: the legacy of {A}. {N}. {K}olmogorov}.
\newblock Cambridge university press, 1995.

\bibitem{galdi2011introduction}
Giovanni Galdi.
\newblock {\em An introduction to the mathematical theory of the Navier-Stokes
  equations: {S}teady-state problems}.
\newblock Springer Science \& Business Media, 2011.

\bibitem{gilbarg1977elliptic}
David Gilbarg, Neil~S Trudinger, David Gilbarg, and NS~Trudinger.
\newblock {\em Elliptic partial differential equations of second order}, volume
  224.
\newblock Springer, 1977.

\bibitem{kato1984remarks}
Tosio Kato.
\newblock Remarks on zero viscosity limit for nonstationary {N}avier-{S}tokes
  flows with boundary.
\newblock In S.~S. Chern, editor, {\em Seminar on Nonlinear Partial
  Differential Equations}, volume~2 of {\em Math. Sci. Res. Inst. Publ.}, pages
  85--98. Springer, 1984.

\bibitem{kelliher2017observations}
James Kelliher.
\newblock Observations on the vanishing viscosity limit.
\newblock {\em Transactions of the American Mathematical Society},
  369(3):2003--2027, 2017.

\bibitem{kress1969grundzuege}
Rainer Kress.
\newblock Grundz\"uge einer {T}heorie der verallgemeinerten harmonischen
  {V}ektorfelder.
\newblock In Bruno Brosowski and Erich Martensen, editors, {\em Methoden und
  {V}erfahren der {M}athematischen {P}hysik, {B}and 2}, volume 721/721a of {\em
  BI {H}ochschulskripten}, pages 49--83. Bibliographisches Institut/Peter D.
  Lang, Mannheim, Wien, Z\"urich, 1969.

\bibitem{lee2013smooth}
John~M Lee.
\newblock {\em Introduction to Smooth Manifolds}.
\newblock Springer, 2013.

\bibitem{leray1934mouvement}
Jean Leray.
\newblock Sur le mouvement d'un liquide visqueux emplissant l'espace.
\newblock {\em Acta mathematica}, 63:193--248, 1934.

\bibitem{lighthill1986informal}
James Lighthill.
\newblock {\em An informal introduction to theoretical fluid mechanics},
  volume~2 of {\em IMA monograph series}.
\newblock Oxford University Press, New York, NY, 1986.

\bibitem{lions1996mathematical}
Pierre-Louis Lions.
\newblock {\em Mathematical topics in fluid mechanics. {V}ol. 1:
  {I}ncompressible models}, volume~3 of {\em Oxford Lecture Series in
  Mathematics and its Applications}.
\newblock Oxford: Clarendon Press, 1996.

\bibitem{nguyenvanyen2018energy}
Natacha Nguyen~van yen, Matthias Waidmann, Rupert Klein, Marie Farge, and Kai
  Schneider.
\newblock Energy dissipation caused by boundary layer instability at vanishing
  viscosity.
\newblock {\em Journal of Fluid Mechanics}, 849:676–717, 2018.

\bibitem{orlandi1990vortex}
Paolo Orlandi.
\newblock Vortex dipole rebound from a wall.
\newblock {\em Physics of Fluids A: Fluid Dynamics}, 2(8):1429--1436, 1990.

\bibitem{prandtl1925magnuseffekt}
Ludwig Prandtl.
\newblock Magnuseffekt und {W}indkraftschiff.
\newblock {\em Naturwissenschaften}, 13(6):93--108, 1925.

\bibitem{prodi1959teorema}
Giovanni Prodi.
\newblock Un teorema di unicita per le equazioni di navier-stokes.
\newblock {\em Annali di Matematica pura ed applicata}, 48:173--182, 1959.

\bibitem{quan2024onsager}
Hao Quan and Gregory~L Eyink.
\newblock Onsager theory of turbulence, the {J}osephson--{A}nderson relation,
  and the {D}’{A}lembert paradox.
\newblock {\em Communications in Mathematical Physics}, 405(11):276, 11 2024.

\bibitem{quan2022inertial}
Hao Quan and Gregory~L Eyink.
\newblock Inertial momentum dissipation for viscosity solutions of {E}uler
  equations: {E}xternal flow around a smooth body.
\newblock submitted to Nonlinearity; \url{arXiv:2206.05325}, 2025.

\bibitem{quan2025weakb}
Hao Quan and Gregory~L. Eyink.
\newblock Weak-strong uniqueness and extreme wall events at high {R}eynolds
  number.
\newblock Phys. Rev. Fluids, submitted, 2025.

\bibitem{scheffer1993inviscid}
Vladimir Scheffer.
\newblock An inviscid flow with compact support in space-time.
\newblock {\em Journal of geometric analysis}, 3(4), 1993.

\bibitem{Serrin1963}
James Serrin.
\newblock The initial value problem for the {Navier-Stokes} equations.
\newblock In {\em Nonlinear Problems (Proc. Sympos., Madison, Wis., 1962)},
  pages 69--98, Madison, Wisconsin, 1963. Univ. of Wisconsin Press.

\bibitem{shnirelman1997nonuniqueness}
Alexander Shnirelman.
\newblock On the nonuniqueness of weak solution of the euler equation.
\newblock {\em Commun. Pure Appl. Math.}, 50(12):1261--1286, 1997.

\bibitem{stein2009real}
Elias~M Stein and Rami Shakarchi.
\newblock {\em Real analysis: measure theory, integration, and Hilbert spaces}.
\newblock Princeton University Press, 2009.

\bibitem{tsai2018lectures}
Tai-Peng Tsai.
\newblock {\em Lectures on Navier-Stokes equations}, volume 192.
\newblock American Mathematical Soc., 2018.

\bibitem{vasseur2023boundary}
Alexis~F Vasseur and Jincheng Yang.
\newblock Boundary vorticity estimates for {N}avier--{S}tokes and application
  to the inviscid limit.
\newblock {\em SIAM Journal on Mathematical Analysis}, 55(4):3081--3107, 2023.

\bibitem{von2006necessary}
Wolf von Wahl.
\newblock On necessary and sufficient conditions for the solvability of the
  equations rot $\mu$=$\gamma$ and div $\mu$=$\epsilon$ with $\mu$ vanishing on
  the boundary.
\newblock In John~G. Heywood, Ky{\^u}ya Masuda, Reimund Rautmann, and
  Vsevolod~A. Solonnikov, editors, {\em The Navier-Stokes Equations Theory and
  Numerical Methods}, pages 152--157, Berlin, Heidelberg, 1990. Springer Berlin
  Heidelberg.

\bibitem{wiedemann2018weak}
Emil Wiedemann.
\newblock Weak-strong uniqueness in fluid dynamics.
\newblock In Charles~L. Fefferman, James~C. Robinson, and Jos{\'e}~L. Rodrigo,
  editors, {\em Partial Differential Equations in Fluid Mechanics}, volume 452
  of {\em London Mathematical Society Lecture Note Series}, page 289–326.
  Cambridge University Press, 2018.

\bibitem{wu1981theory}
James~C Wu.
\newblock Theory for aerodynamic force and moment in viscous flows.
\newblock {\em AIAA Journal}, 19(4):432--441, 1981.

\end{thebibliography}
\end{document}